\documentclass{article}

\usepackage{amsmath, amsthm, amssymb}
\usepackage{xspace}
\usepackage{hyperref}

\newtheorem{thm}{Theorem}[section]
\newtheorem{lem}[thm]{Lemma}
\newtheorem{slem}[thm]{Sublemma}
\newtheorem{rmk}[thm]{Remark}
\newtheorem{cor}[thm]{Corollary}
\newtheorem{defn}[thm]{Definition}
\newtheorem{mainthm}{Theorem}

\newcommand{\N}{\mathbb{N}}
\newcommand{\R}{\mathbb{R}}
\newcommand{\Z}{\mathbb{Z}}
\newcommand{\mcl}{\mathcal L}
\newcommand{\bbp}{\mathbb P}

\newcommand{\al}{\alpha}

\newcommand{\ga}{\gamma}

\newcommand{\del}{\delta}
\newcommand{\ep}{\epsilon}
\newcommand{\sig}{\sigma}
\newcommand{\ka}{\kappa}
\newcommand{\lam}{\lambda}
\newcommand{\Lam}{\Lambda}
\newcommand{\Om}{\Omega}
\newcommand{\om}{\omega}

\newcommand{\tld}[1]{\tilde{#1}}
\newcommand{\mc}[1]{\mathcal{#1}}

\newcommand{\leb}{\text{Leb}}

\newcommand{\essinf}{\mathop{\mathrm{essinf}}}
\newcommand{\var}{\mathop{\mathrm{var}}}

\newcommand{\SOT}{\text{SOT}}

\newcommand{\pvu}[1]{\Pi_{V_+\| U \oplus U_-(#1)}}
\newcommand{\puv}[1]{\Pi_{U\|V_+\oplus U_-(#1)}}
\newcommand{\pvun}[1]{\pvu{\sig^{#1}\om}}
\newcommand{\puvn}[1]{\puv{\sig^{#1}\om}}

\newcommand{\hpt}{{\mc{H}_p^t}}
\newcommand{\lhpt}{{H_p^t}}
\newcommand{\LY}{\text{LY}} % Lasota Yorke maps
\newcommand{\PE}{\text{PE}} % Piecewise expanding maps

\newcommand{\G}{\mc{G}}
\newcommand{\nint}{b} % number of intervals of monotonicity/cont of LY map
\newcommand{\dist}{\mc{D}} % distortion bound in random LY
\newcommand{\tn}{\ensuremath{| \! | \! |}}

\author{Cecilia Gonz\'alez-Tokman\thanks{Department of Mathematics and Statistics, University of Victoria,
Victoria, B.C., Canada V8W 3R4. Email: ceciliag@uvic.ca.
CGT is supported by a PIMS Postdoctoral Fellowship.}
\and
Anthony Quas \thanks{Department of Mathematics and Statistics, University of Victoria,
Victoria, B.C., Canada V8W 3R4. Email: aquas@uvic.ca. 
AQ is supported by NSERC.}
}
\title{A semi-invertible operator Oseledets theorem}

\begin{document}
\maketitle

\begin{abstract}
  Semi-invertible multiplicative ergodic theorems establish the
  existence of an Oseledets splitting for cocycles of non-invertible
  linear operators (such as transfer operators) over an invertible
  base.  Using a constructive approach, we establish a semi-invertible
  multiplicative ergodic theorem that for the first time can be
  applied to the study of transfer operators associated to the
  composition of piecewise expanding interval maps randomly chosen
  from a set of cardinality of the continuum.  We also give an
  application of the theorem to random compositions of perturbations
  of an expanding map in higher dimensions.
\end{abstract}

\section{Introduction}

\subsection{Motivation and History}
Oseledets' proof, in 1965, of the multiplicative ergodic theorem is a
milestone in the development of modern ergodic theory. It has been
applied to differentiable dynamical systems to establish the existence of
Lyapunov exponents and plays a crucial role in the construction of stable
and unstable manifolds. It also has substantial applications in the
theory of random matrices, Markov chains, etc.

The proof has been generalized in many directions by a number of
authors (including Ruelle \cite{Ruelle}, Ma\~n\'e \cite{Mane},
Ledrappier \cite{KarlssonLedrappier}, Raghunathan \cite{Raghunathan},
Kai\-ma\-no\-vich \cite{Kaimanovich} and many others). In the original
version, one has an ergodic measure-preserving system $\sigma\colon
\Omega\to \Omega$ and for each $\omega\in \Omega$, a corresponding
matrix $A(\omega)\in M_d(\mathbb R)$. Under suitable integrability
conditions on the norms of the matrices it is shown that over almost
every point, $\omega$, of $\Omega$, there is a measurably-varying
collection of subspaces $(V_i(\omega))_{1\le i\le k}$, with a
decreasing sequence of characteristic exponents $\lambda_i$ such that
(i) the subspaces are \textsl{equivariant} - that is,
$A(\omega)(V_i(\omega))\subset V_i(\sigma\omega)$; and (ii) that
vectors in $V_i(\omega)$ (typically) expand at rate $\lambda_i$ under
sequential applications of the matrices $A(\sigma^j\omega)$ along the
orbit. That is,
\[
\lim_{n\to \infty} \frac{1}{n}\log\|A(\sigma^{n-1}\omega)\cdots
A(\omega)v\| =\lambda_i.
\]

More specifically and of crucial significance for this article,
Oseledets' multiplicative ergodic theorem was proved in two versions: an
invertible version and a non-invertible version.

In the invertible version the following is assumed: $\sigma$ is an
invertible transformation of $\Omega$; the matrices $A(\omega)$ are each
invertible and $\int\log^+\|A(\omega)\|\,d\mathbb P(\omega)$ and
$\int\log^+\|A(\omega)^{-1}\|\,d\mathbb P(\omega)$ are both finite. The
conclusion of the theorem is then that there is for almost every $\omega$
a measurable \textsl{splitting} of $\mathbb R^d$:
\begin{equation}\label{eq:splitting}
  \mathbb R^d=Y_1(\omega)\oplus Y_2(\omega)\oplus\ldots\oplus Y_l(\omega)
\end{equation}
such that for all $v\in Y_i(\omega)\setminus\{0\}$
\begin{align}
  \lim_{n\to\infty}
  \frac{\log\|A^{(n)}(\omega)v\|}n&=\lambda_i;\label{rate:fwd}\\
  \lim_{n\to-\infty}
  \frac{\log\|A^{(n)}(\omega)v\|}n&=\lambda_i,\label{rate:backwd}
\end{align}
where $A^{(n)}$ denotes the matrix cocycle $A(\sigma^{n-1}\omega)\cdots
A(\omega)$ for $n>0$ whereas for $n<0$, it is
$A(\sigma^{-n}\omega)^{-1}\cdots A(\sigma^{-1}\omega)^{-1}$.

In the non-invertible version of the theorem, $\sigma$ is no longer
assumed to be invertible and there is no assumption on the invertibility
of the matrices $A(\omega)$. In this case there is a weaker conclusion:
rather than a splitting of $\mathbb R^d$ one obtains a
\textsl{filtration}: A decreasing sequence of subspaces of $\mathbb R^d$
\begin{equation}\label{eq:filtration}
  \mathbb R^d=V_1(\omega)\supset V_2(\omega)\supset\ldots\supset
  V_l(\omega)
\end{equation}
such that for all $v\in V_i(\omega)\setminus V_{i+1}(\omega)$ (defining
$V_{l+1}(\omega)$ to be $\{0\}$), \eqref{rate:fwd} holds.

In \cite{FLQ1}, Froyland, Lloyd and Quas refined the dichotomy between
invertible and non-invertible versions of the theorem, introducing a
third class of versions of the theorem: \textsl{semi-in\-vert\-ible}
multiplicative ergodic theorems. For semi-in\-vert\-ible ergodic theorems the
underlying dynamical system is assumed to be invertible, but no
assumption is made on the invertibility of the matrices. The conclusion
of the theorem in this category is that there is again a splitting of the
vector space (instead of a filtration) and that for all $v\in
Y_i(\omega)\setminus\{0\}$, \eqref{rate:fwd} holds (but not
\eqref{rate:backwd} which does not make sense in this context).

Our motivation for considering semi-in\-vert\-ible multiplicative
ergodic theorems comes from application-oriented studies of rates of
mixing due to Dellnitz, Froyland and collaborators
\cite{DellnitzJunge, DellnitzFroylandSertl, Froyland}.  Given a
measure-preserving dynamical system it is called (strong-) mixing if
$\mu(A\cap T^{-n}B)\to\mu(A)\mu(B)$ for all measurable sets $A$ and
$B$. This is an asymptotic independence property for any measurable
sets under evolution.

An equivalent formulation of mixing is that $\int f\cdot g\circ
T^n\,d\mu$ should converge to $\int f\,d\mu\int g\,d\mu$ for all $L^2$
functions $f$ and $g$. Clearly nothing is lost if one demands that the
functions should have zero integral.

Relaxing the assumption that $\mu$ is an invariant measure, one may take
$\mu$ to be some ambient measure (e.g. Lebesgue measure in the case that
$T$ is a smooth map of a manifold or subset of $\mathbb R^d$). A key tool
in this study is the \textsl{Perron-~Frobenius Operator} or
\textsl{transfer operator}, $\mathcal L$, acting on $L^1(\mu)$. This is
the pre-dual of the operator of composition with $T$ acting on
$L^\infty$, the so-called Koopman operator, so that $\int f\cdot g\circ
T\,d\mu=\int \mathcal Lf\cdot g\,d\mu$ for all $f\in L^1$ and $g\in
L^\infty$). In many cases one can give a straightforward expression for
$\mathcal Lf$. It is not hard to check from the definition that $\mathcal
Lf=f$ if and only if $f$ is the density of an absolutely continuous
invariant measure for $T$.

One might naively ask for the rate of convergence of $\int f\cdot g\circ
T^n\,d\mu$ to 0 if indeed the system is mixing, but simple examples show
that there is no uniform rate of convergence: one can construct in any
non-trivial mixing system, functions $f$ and $g$ such that the rate of
convergence is arbitrarily slow. One does however obtain rates of mixing
if one places suitable restrictions on the class of `observables' $f$ and
$g$ for which one computes $\int f\cdot g\circ T^n\,d\mu$. It turns out
that an important reason that the Perron-Frobenius operator is so useful
is that if one restricts the function $f$ to a suitable smaller Banach
space of observables $B\subset L^1$, then in many cases $\mathcal L$ maps
$B$ to $B$; and better still $\mathcal L$ is a quasi-compact operator on
$B$, so that the spectrum of $\mathcal L$ consists of a discrete set of
values outside the essential spectral radius each corresponding to
eigenvalues of $\mathcal L$ with finite-dimensional eigenspaces. Given
this one can relate the rate of mixing of the dynamical system
(restricted to a suitable class of observables) to the spectral
properties of the operator $\mathcal L$ restricted to the Banach space
$B$. It is a key fact for our purposes that the Perron-Frobenius
operators $\mathcal L$ that one works with are almost invariably
non-invertible.

Ulam's method takes this one step further, replacing the operator
$\mathcal L$ by a finite rank approximation. In works of Froyland
\cite{FroylandUlam1D, FroylandUlamApprox} and Baladi, Isola and
Schmitt \cite{BaladiIsolaSchmitt}, the relationship between the finite
rank approximations of $\mathcal L$ and the original Perron-Frobenius
operators is studied. This turns out to be remarkably effective and
this is a good technique for computing invariant measures numerically
(see for example work of Dellnitz and Junge \cite{DellnitzJunge},
Froyland\cite{FroylandHigherDim}, Keane, Murray and Young
\cite{KeaneMurrayYoung}).  Keller and Liverani \cite{KellerLiverani99}
showed that exceptional eigenvalues of $\mathcal L$ (those outside the
essential spectral radius) persist under approximation of $\mathcal
L$.

In a development of Ulam's method, \cite{DellnitzJunge} and later
\cite{FroylandUlamApprox} related the large sub-unit eigenvalues and
corresponding eigenvectors of the finite rank approximation of
$\mathcal L$ to properties of the underlying system. In particular
they showed that these exceptional eigenvectors give rise to global
features inhibiting mixing of the system (whereas the essential
spectral radius is related to local features inhibiting mixing of the
system). For a cartoon example, one can consider a map of the interval
$[-1,1]$ in which the left sub-interval $[-1,0]$ and right
sub-interval $[0,1]$ are almost invariant (that is only a small amount
of mass leaks from one to the other under application of the map) but
within each subinterval there is rapid mixing- see work of
Gonz\'alez-Tokman, Hunt and Wright \cite{GTHuntWright} and Dellnitz,
Froyland and Sertl \cite{DellnitzFroylandSertl}. In this case one
observes eigenvalues that are close to 1, where the eigenfunction
takes values close to 1 on one sub-interval and close to $-1$ on the
other sub-interval. In applied work Dellnitz, Froyland and
collaborators\cite{FroylandPadbergEnglandTreguier,
  DellnitzFroylandEtal} make use of these exceptional eigenvectors to
analyse the ocean and locate regions with poor mixing, called gyres.

The current work (and its predecessors \cite{FLQ1} and \cite{FLQ2}) is
motivated by extending the program of Dellnitz and Froyland to the
case of forced dynamical systems (or equivalently random dynamical
systems), that is systems of the form
$T(\omega,x)=(\sigma(\omega),T_\omega(x))$.  Again as a cartoon
example, one can consider the effect of the moon on the oceans: the
moon evolves autonomously (and \textsl{invertibly}), whereas the
evolution of the ocean is affected by the position of the
moon. Relating this to the context of the multiplicative ergodic
theorem, think of the dynamical system $\sigma\colon\Omega\to\Omega$
as being the autonomous dynamics of the moon and the
$\omega$-dependent matrix to be a map on a Banach space of densities
in the ocean. The aim is, once again, to identify and study the second
and subsequent exceptional eigenspaces with a view to understanding
obstructions to mixing. The importance of the semi-in\-vert\-ible
multiplicative ergodic theorems here (the underlying base dynamics is
invertible but the Perron-Frobenius operators are non-invertible) are
that the obstructions to mixing, the $V_2(\omega)$, appear here as
finite-dimensional subspaces rather than the finite-codimensional
subspaces that one would obtain from the standard multiplicative
ergodic theorems. This program has been demonstrated to work in
practice for driven cylinder flows in an article of Froyland, Lloyd
and Santitissadeekorn \cite{FroylandLloydSantitissadeekorn}.

In all three works, this paper and its two predecessors, \cite{FLQ1} and
\cite{FLQ2}, the goal is to prove a semi-in\-vert\-ible multiplicative
ergodic theorem and apply it to as general a class of random dynamical
systems as possible. In all three papers, the starting point was a pair
of multiplicative ergodic theorems: an invertible and a non-invertible;
and then to derive, using the pair of ergodic theorems as \textsl{black
boxes}, a semi-in\-vert\-ible ergodic theorem.

\cite{FLQ1} dealt with the original Oseledets context of $d\times d$
real matrices (and used Oseledets' original theorem \cite{Oseledets}
as the basis). \cite{FLQ2} dealt with the case of an operator-valued
multiplicative ergodic theorem where the map $\mathcal L\colon
\omega\mapsto \mathcal L(\omega)$ is (almost)-continuous with respect
to the operator norm (using a theorem of Thieul\-len \cite{Thieullen}
as a basis). The current paper deals with the case of an
operator-valued multiplicative ergodic theorem where the map
$\omega\mapsto\mathcal L(\omega)$ is measurable with respect to a
$\sigma$-algebra related to the strong operator topology (using a
Theorem of Lian and Lu \cite{LianLu} as a basis).

The applications to random dynamical systems have become progressively
more general through the sequence of works: \cite{FLQ1} applied to
finite-dimensional approximations of random dynamical systems (using the
Ulam scheme) as well as dealing exactly with some dynamical systems
satisfying an extremely strong jointly Markov condition. \cite{FLQ2}
applied to one-dimensional expanding maps. However, since the set of
Perron-Frobenius operators of $C^2$ expanding maps acting on the space of
functions of bounded variation is uniformly discrete, the conditions of
the theorem restricted the authors to studying random dynamical systems
with at most countably many maps. In the current paper, Lian and Lu's
result allows us to weaken the continuity assumption to strong
measurability (defined below). Essentially, this amounts to checking
continuity of $\omega\mapsto \mathcal L_{T_\omega}f$ for a \textsl{fixed}
$f$. The cost, however, is that the Banach space on which the transfer
operators act is now required to be separable (which the space of
functions of bounded variation, used in \cite{FLQ2}, is not). In order to
apply the semi-in\-vert\-ible ergodic theorem to random one-dimensional
expanding maps, we make substantial use of recent work of Baladi and
Gou\"ezel \cite{BaladiGouezel} who used a family of local Sobolev norms
to study Perron-Frobenius operators of (higher-dimensional) piecewise
hyperbolic maps; see also Thomine \cite{Thomine} for a specialization in
the context of expanding maps. While Baladi and Gou\"ezel were working
with a single map, we show that the Perron-Frobenius operators on the
Banach spaces that they construct depend in a suitable way for families
of expanding one-dimensional maps allowing us to apply our
semi-in\-vert\-ible multiplicative ergodic theorem (making essential use
also of an idea of Buzzi \cite{Buzzi}). We also point out that, to our
knowledge, it was not even known whether an Oseledets filtration existed
in this setting.

Another feature of the proofs is the way in which the semi-in\-vert\-ible
theorem is proved from the invertible and non-invertible theorems. The
essential issue is that the non-invertible theorem provides equivariant
families of finite co-dimensional subspaces $V_i(\omega)$ (being the set
of vectors that expand at rate $\lambda_i$ or less). One is then
attempting to build an equivariant family of (finite-dimensional) vector
spaces $Y_i(\omega)$ so that $V_{i+1}(\omega)\oplus
Y_i(\omega)=V_i(\omega)$.

In \cite{FLQ1} this was done in a relatively natural way (by pushing
forward the orthogonal complement of $V_i(\sigma^{-n}\omega)\ominus
V_{i+1}(\sigma^{-n}\omega)$ under $A(\sigma^{-1}\omega)\cdots
A(\sigma^{-n}\omega)$ and taking a limit as $n$ tends to infinity).

In \cite{FLQ2}, the proof exploited the structure of the proof given by
Thieullen. Spec\-ifically Thieul\-len first proved the invertible
multiplicative ergodic theorem and then obtained the non-invertible
theorem as a corollary by building an inverse limit Banach space
(reminiscent of the standard inverse limit constructions in ergodic
theory). The finite-co-dimensional family $V_i(\omega)$ was obtained by
projecting the corresponding subspaces from the invertible theorem onto
their zeroth coordinate. In \cite{FLQ2} it was proved that applying the
same projection to the finite-dimensional complementary family yielded
the $Y_i(\omega)$ spaces. This proof, while relatively simple, is
problematic for applications as there appears to be no sensible way to
computationally work with these inverse limit spaces. We see this proof
technique as non-constructive. This non-constructive proof technique
should probably apply with a high degree of generality.

In the current paper we come back much closer to the scheme applied in
\cite{FLQ1}. The same non-constructive techniques that were used by
Thieul\-len to obtain the non-invertible theorem from the invertible
theorem were used by Doan in his thesis \cite{Son} to obtain a
non-invertible version of the result of Lian and Lu \cite{LianLu}.
Starting from the non-constructive existence proof of the finite
co-dimensional subspaces we obtain a constructive proof of the
finite-dimensional $Y_i(\omega)$ spaces. We see this as being likely to
lead to computational methods although we have not implemented these at
the current time.

\subsection{Statement of Results and structure of paper}

The context of Lian and Lu's multiplicative ergodic theorem is that of
strongly measurable families of operators.

If $X$ is a separable Banach space, then $L(X)$ will denote the set of
bounded linear maps from $X$ to $X$. The strong operator topology on
$L(X)$ is the topology generated by the sub-base consisting of sets of
the form $\{T\colon \|T(x)-y\|<\epsilon\}$. The strong $\sigma$-algebra
$\mathcal S$ is defined to be the Borel $\sigma$-algebra on $L(X)$
generated by the strong operator topology. Appendix A develops a number
of basic results about strong-measurability, including the following
useful characterization: A map $\mathcal L\colon \Omega\to L(X)$ is
strongly measurable if for each $x\in X$, the map $\Omega\to X$,
$\omega\mapsto \mathcal L(\omega)(x)$ is measurable with respect to the
$\sigma$-algebra on $\Omega$ and the Borel $\sigma$-algebra on $X$.

Of course the `strong operator topology' is very much coarser than the
norm topology on $L(X)$ - checking continuity in the strong operator
topology can be done one $x$ at a time. This is the essential difference
between the result of Thieul\-len and that of Lian and Lu: for a given
function $f$, $\mathcal L(\omega)f$ and $\mathcal L(\omega')f$ are close
if $T_\omega$ and $T_{\omega'}$ are close enough, but the operators
$\mathcal L_{T_\omega}$ and $\mathcal L_{T_{\omega'}}$ are, in many
interesting cases, uniformly far apart. (An exception to this is the
setting of smooth expanding analytic maps.)

For convenience we state our main results here, even though some of the
terms in the statement have yet to be defined. These correspond to
Theorems \ref {thm:GranderOseledetsSplitting} and
\ref{thm:OseledetsSplitting4LYMaps} in the body of the paper.

Our new semi-in\-vert\-ible multiplicative ergodic theorem is the
following (for simplicity we state the version in which there are
finitely many exceptional exponents; a corresponding version holds if
there are countably many exponents which then necessarily converge to
$\kappa^*$).

\begin{mainthm}\label{mainthm:A}
Let $\sigma$ be an invertible ergodic measure-preserving transformation
of the Lebesgue space $(\Omega,\mathcal F,\mathbb P)$. Let $X$ be a
separable Banach space. Let $\mathcal L\colon
\Omega\to L(X)$ be a strongly measurable family of mappings such that
$\log^+\|\mathcal L(\omega)\|\in L^1(\mathbb P)$ and suppose that the
random linear system $\mathcal R=(\Omega,\mathcal F,\mathbb
P,\sigma,X,\mathcal L)$ is quasi-compact (i.e. the analogue of the
spectral radius, $\lambda^*$, is larger than the analogue of the
essential spectral radius, $\kappa^*$).

Then there exists $1\le l\le \infty$ and a sequence of exceptional
Lyapunov exponents
$\lam^*=\lambda_1>\lambda_2>\ldots>\lambda_l>\kappa^*$ (or in the case
$\lambda=\infty$, $\lam^*=\lambda_1>\lambda_2>\ldots$;
$\lim_{n\to\infty} \lambda_n=\kappa^*$).

For $\mathbb P$-almost every $\omega$ there exists a unique measurable
equivariant splitting of $X$ into closed subspaces
$X=V(\omega)\oplus\bigoplus_{j=1}^l Y_j(\omega)$ where the $Y_j(\omega)$
are finite-dimensional. For each $y\in Y_j(\omega)\setminus\{0\}$,
$\lim_{n\to\infty}\frac 1n\log\|\mathcal L_\omega^{(n)}y\|=\lambda_j$.
For $y\in V(\omega)$, $\lim_{n\to\infty}\frac 1n\log\|\mathcal
L_\omega^{(n)}y\|\le \kappa^*$.
\end{mainthm}

The application to random piecewise expanding systems is as follows:

\begin{mainthm}\label{mainthm:B}
  Let $\sigma$ be an invertible ergodic measure-preserving
  transformation of the Lebesgue space $(\Omega,\mathcal F,\mathbb
  P)$. For each $\omega\in\Omega$, let $T_\omega$ be a random
  expanding dynamical system acting on $X_0\subset \R^d$.
  Assume further that
  $\omega\mapsto T_\omega$ is Borel-measurable, the $C^{1+\alpha}$ norm of
  $T_\omega$ is uniformly bounded above, the maps
  $T_\omega$ have a derivative that is uniformly bounded away from 1,
  and that some integrability conditions are satisfied.

  Suppose that either $d=1$ (Lasota-Yorke case);
  or $d>1$ and the maps $T_\omega$ are $C^2$, have a common branch
  partition and belong to a sufficiently small neighbourhood of a
  Cowieson map.

Then there exist a separable, reflexive Banach space $X$ containing the
$C^\infty$ functions supported on $X_0$ for which the map $\om \mapsto \mcl_\om$
given by the transfer operator associated to $T_\om$ is strongly
measurable, a quantity $1\le l\le\infty$, a sequence of exceptional
exponents $0=\lambda_1>\ldots>\lambda_l>\kappa^*$, (or if $l=\infty$,
then $0=\lambda_1>\lambda_2>\ldots>\kappa^*$; $\lim_{n\to\infty}
\lambda_n=\kappa^*$), and a family of finite-dimensional equivariant
subspaces $(Y_i(x))_{1\le i\le l}$ satisfying the conclusions of Theorem
\ref{mainthm:A}.
\end{mainthm}

The main motivation behind our search for semi-invertible Oseledets
theorems has been to provide a general framework in which it is
possible to identify low-dimensional spaces that are responsible for
impeding mixing in infinite-dimensional dynamical systems. Following
Dellnitz, Froyland and collaborators we want to extract information
not simply from the exceptional Lyapunov exponents, but rather from
the corresponding Lyapunov subspaces.

It is important to note that exponential decay of correlations is not
assumed. Our work applies, for instance, to an example of Buzzi in
\cite{BuzziEDC} (Example 3). Buzzi's example (which works by essentially
having two copies of the interval and a pair of maps each of which acts
as doubling on each interval and then simply permutes the intervals)
fails to have exponential decay of correlations, but it is still
quasi-compact. In our context this will be reflected in the fact that the
top exceptional Lyapunov subspace has multiplicity 2 rather than 1. In
fact, the structure of this top subspace exactly illustrates the goal of
our work because the Oseledets space will consist of a constant function and a
function which is 1 on one of the intervals and $-1$ on the other,
thereby indicating the source of non-mixing.

In addition, there are examples in the existing literature showing the possibilities of 
having more than one Oseledets space; that is, $l\geq 2$.  In the random setting, there is an example by Froyland, Lloyd and Quas, \cite[Theorem
5.1]{FLQ1}; in the deterministic case, there is one by Keller and Rugh \cite[Theorem 1]{KellerRugh}.
In fact, it is \textsl{a priori} possible to have all sorts of combinations for
number of exceptional Lyapunov exponents ($1\leq l \leq \infty$) and
multiplicities ($1\leq m_1, \dots, m_l <\infty$), in a similar way that
square matrices may have different Jordan normal forms. 

In section \ref{sec:OseledetsSplitting} we give the proof of the
semi-in\-vert\-ible multiplicative theorem. In
section~\ref{sec:RandomLY} we introduce the fractional Sobolev spaces
(as used in Baladi and Gou\"ezel) and study the continuity properties
of the map sending a Lasota-Yorke map to its Perron-Frobenius
operator. We then adapt the proof given by Baladi and Gou\"ezel of
quasi-compactness for a single map to the situation of a random
composition of one-dimensional expanding maps (using results of
Hennion and Buzzi) to show that the theorem of section
\ref{sec:OseledetsSplitting} applies in this context. We also present an application of Theorem~\ref{mainthm:A} to piecewise expanding maps in higher dimensions, building on work of Cowieson \cite{Cowieson}.
Section~\ref{sec:futureWork} summarizes possible directions for future work.

The paper has three appendices: Appendix A contains results about strong
measurability. Appendix B contains results about the Grassmannian of a
separable Banach space. Appendix C collects some results from ergodic
theory: a useful characterization of tempered maps and a Hennion type
theorem for random linear systems.

\subsection{Acknowledgments}
The authors would like to thank Viviane Baladi, Chris Bose, J\'er\^ome
Buzzi, Jim Campbell, and Karl Petersen for useful discussions, as well as
Gary Froyland and an anonymous referee for providing a wide range of
helpful suggestions and bibliographical references.
 % Intro
\section{Oseledets splittings for random linear systems.}\label{sec:OseledetsSplitting}

\subsection{Preliminaries}

We start by introducing some notation about
random dynamical systems.

\begin{defn}\label{defn:RandomDS}
  A \textbf{separable strongly measurable random linear system} is a tuple
  $\mc{R}=(\Om,\mc{F}, \bbp, \sig, X, \mcl)$ such that $(\Omega,\mathcal
  F,\mathbb P)$ is a Lebesgue space, $\sigma$ is a probability preserving
  transformation of $(\Omega,\mathcal F,\mathbb P)$, $X$ is a separable Banach space, and the generator
  $\mcl: \Om \to L(X)$ is a strongly measurable map (see
  Definition~\ref{defn:stronglyMble}).  We use the
  notation $\mcl_\om^{(n)}=\mcl(\sig^{n-1}\om) \circ \dots \circ \mcl(\om)$.
\end{defn}

\begin{defn}\label{defn:indexCompactness}
  The \textbf{index of compactness} (or Kuratowski measure of
  non-com\-pact\-ness) of a bounded linear map $A: X\to X$ is
\[
\|A\|_{ic(X)}=\inf \{r>0 : A(B_X) \text{ can be covered by finitely
  many } \text{ balls of radius } r \},
\]
where $B$ denotes the unit ball in $X$.
\end{defn}

\begin{defn} \label{defn:MLE+IC} Let $\mc{R}=(\Om,\mc{F}, \bbp, \sig, X,
  \mcl)$ be a separable strongly measurable random linear system.  Assume that
  $\int \log^+ \|\mcl_\om\| d\bbp(\om)<\infty$.  For each $\om \in \Om$,
  the \textbf{maximal Lyapunov exponent} for $\om$ is defined as
\[
  \lam(\om):=\lim_{n \to \infty} \frac{1}{n} \log \|\mcl_{\om}^{(n)}\|,
\]
whenever the limit exists.  For each $\om \in \Om$, the \textbf{index
  of compactness} for $\om$ is defined as
\[
  \ka(\om):= \lim_{n\to \infty} \frac{1}{n} \log \|\mcl_\om^{(n)}\|_{ic(X)},
\]
whenever the limit exists. Whenever we want to emphasize the
dependence on $\mc{R}$, we will write $\lam_\mc{R}(\om)$ and
$\ka_\mc{R}(\om)$.
 \end{defn}

\begin{lem}\label{lem:MLexpMble}
  Let $\mathcal R$ be as in Definition \ref{defn:MLE+IC}.
  $\lam(\om)$ is well defined for $\bbp$-almost every $\om$. The
  function $\om \mapsto \lam(\om)$ is measurable and $\sig$-invariant.
\end{lem}
\begin{proof}
  The sequence of functions $\{\log \|\mcl_{\om}^{(n)}\|\}_{n\in \N}$ is
  subadditive. That is,
\[
\log \|\mcl_{\om}^{(m+n)}\| \leq \log \|\mcl_{\sig^n \om}^{(m)}\| +\log
\|\mcl_{\om}^{(n)}\|.
\]
Since the composition of strongly measurable maps is strongly measurable
by Lemma~\ref{lem:CompStronglyMeas}, and the sets $L_r(X)=\{A \in L(X):
\|A\|\leq r \}$ are $\mc{S}$ measurable by
Lemma~\ref{lem:Smeas}\eqref{it:BallIsSmeas}, then the map $\om \mapsto
\|\mcl_\om^{(n)}\|$ is measurable. Measurability of the map $\om \mapsto
\lim_{n \to \infty}\sup \frac{1}{n} \log \|\mcl_{\om}^{(n)}\|$ follows.
By Kingman's subadditive theorem \cite{Kingman}, the limit $\lim_{n \to
\infty} \frac{1}{n} \log \|\mcl_{\om}^{(n)}\|$ exists for $\bbp$-almost
every $\om \in \Om$, and it is $\sig$-invariant.
\end{proof}

\begin{lem}
  Let $\mathcal R$ be as in Definition \ref{defn:MLE+IC}.
  The index of compactness is finite, submultiplicative and
  measurable, when $L(X)$ is equipped with the strong $\sig$-algebra
  $\mc{S}$. Thus, $\ka(\om)$ is well defined for $\bbp$-almost every
  $\om$.  The function $\om \mapsto \ka(\om)$ is measurable and
  $\sig$-invariant.
\end{lem}

\begin{proof}
  The index of compactness is bounded by the norm. Submultiplicativity
  is straightforward to check.  To show $\mc{S}$-measurability of the
  index of compactness, we present the argument given in Lian and Lu
  \cite{LianLu}. Let $\{x_i\}_{i\in \N}$ be a dense subset of $X$ and
  $\{y_j\}_{j\in \N}$ be a dense subset of $B(X)$. Let $U$ be the
  (countable) set of finite subsets of $\{x_i\}_{i\in \N}$. Let $U=
  \bigcup_{i \in \N} U_i$. Then, one can check that
\[
\{A \in L(X): \|A\|_{ic(X)}<r\} =\bigcup_{n=2}^\infty
\bigcup_{i=i}^\infty \bigcap_{j=1}^{\infty} \bigcup _{x \in U_i} \{ A:
\| A(y_j) - x \| < (1-1/n) r \},
\]
(see \cite[Lemma 6.5]{LianLu} for a proof.)  Hence, $A \mapsto
\|A\|_{ic(X)}$ is $\mc{S}$-measurable.  Thus, $\bbp$-almost everywhere
existence, measurability and $\sig$-invariance of $\ka$ follow just like
in the proof of Lemma~\ref{lem:MLexpMble}.
\end{proof}

\begin{rmk}\label{rmk:MaxLExpICConstant}
  If $\mc{R}$ has an ergodic base, then $\lam$ and $\ka$ are
  $\bbp$-almost everywhere constant. We call these constants
  $\lam^*(\mc{R})$ and $\ka^*(\mc{R})$, or simply $\lam^*$ and $\ka^*$
  if $\mc{R}$ is clear from the context. It follows from the
  definitions that $\ka^* \leq \lam^*$.  The assumption $\int \log^+
  \|\mcl_\om\| d\bbp(\om)<\infty$ implies that $\lam^*<\infty$.
\end{rmk}

\begin{defn}
  A strongly measurable random linear system with ergodic base is
  called \textbf{quasi-compact} if $\ka^*<\lam^*$.
\end{defn}

\subsection{Construction of Oseledets splitting}
The following theorem was obtained by Doan  \cite{Son} as
a corollary of the two-sided Oseledets theorem proved by Lian and Lu
\cite{LianLu}.

\begin{thm}[Doan \cite{Son}]\label{thm:Son}
Let $\mc{R}=(\Omega,\mathcal F,\sigma,\mathbb P, X, \mcl)$ be a separable strongly
measurable random linear system with ergodic base. Assume that
$\log^+\|\mcl(\omega)\|\in L^1(\Omega,\mathcal F,\mathbb P)$ and that
$\mc{R}$ is quasi-compact.  Then, there exists $1\leq l \leq \infty$,
numbers $\lam^*=\lam_1>\dots > \lam_l> \ka^*$ (or in the case $l=\infty$,
$\lam_1>\lam_2>\ldots>\ka^*$; $\lim_{n\to\infty}\lam_n=\ka^*$), the
exceptional Lyapunov exponents of $\mc{R}$, multiplicities $m_1, \dots,
m_l$, and a filtration $X=V_1(\omega)\supset\ldots\supset
V_l(\omega)\supset V_{l+1}(\omega)$ of finite-codimensional subspaces (in
the case $l=\infty$ we have $V_1(\omega)\supset
V_2(\omega)\supset\ldots$) defined on a full $\bbp$-measure,
$\sig$-invariant subset of $\Om$ satisfying:

\begin{enumerate}
\item For every $1\leq j \leq l$, $\mcl_\om V_j(\om)\subseteq V_{j}(\sig
  \om)$, and the codimension of $V_{j+1}(\om)$ in $V_j(\om)$ is
  $m_j$.
  Furthermore, $\mcl_\om V_{l+1}(\om)\subseteq V_{l+1}(\sig \om)$.
\item
For every $1\leq j \leq l$ and $v\in V_{j}(\om)\setminus V_{j+1}(\om)$,
\[
  \lim_{n \to \infty} \frac{1}{n}\log \|\mcl_{\om}^{(n)}v\|=\lam_j.
\]
For every $v\in V_{l+1}(\om)$,
\[
  \limsup_{n \to \infty} \frac{1}{n}\log \|\mcl_{\om}^{(n)}v\|\leq \ka^*.
\]
\end{enumerate}
\end{thm}

\begin{rmk}\label{rmk:FiltrationIsMble}
  Combining the result of Lian and Lu \cite{LianLu} with
  Lemma~\ref{lem:ImageProj} and the proof of \cite{Son}, we obtain
  that the spaces $V_j(\om)$ forming the filtration given by
  Theorem~\ref{thm:Son} depend measurably on $\om$.
\end{rmk}

The main result of this section is the following.
\begin{thm}[Semi-invertible operator Oseledets
  theorem]\label{thm:GranderOseledetsSplitting}\ \\
  Let $\mc{R}=(\Om,\mc{F}, \bbp, \sig, X, \mcl)$ be a separable strongly measurable
  random linear system with ergodic invertible base. Assume that
  $\log^+\|\mcl(\omega)\|\in L^1(\Omega,\mathcal F,\mathbb P)$ and that
  $\mc{R}$ is quasi-compact.  Let $\lam^*=\lam_1>\dots > \lam_l>
  \ka^*$ be the exceptional Lyapunov exponents of $\mc{R}$, and $m_1,
  \dots, m_l\in \N$ the corresponding multiplicities
  (or in the case $l=\infty$, $\lam_1>\lam_2>\ldots$ with $m_1,m_2,\ldots$ the
  multiplicities).

  Then, up to $\bbp$-null sets, there exists a unique, measurable,
  equivariant splitting of $X$ into closed subspaces, $X=V(\om)\oplus
  \bigoplus_{j=1}^l Y_j(\om)$, where possibly $V(\om)$ is infinite
  dimensional and $\dim Y_j(\om)=m_j$.  Furthermore, for every $y\in
  Y_j(\om)\setminus \{0\}$, $\lim_{n \to \infty} \frac{1}{n}\log
  \|\mcl_{\om}^{(n)}y\|=\lam_j$, for every $v\in V(\om)$, $ \limsup_{n
    \to \infty} \frac{1}{n}\log \|\mcl_{\om}^{(n)}v\|\leq \ka^*$ and the
  norms of the projections associated to the splitting are tempered
  with respect to $\sig$ (where a function $f:\Om \to \R$ is called
    tempered if for $\bbp$-almost every $\om$,
    $\lim_{n \to \pm \infty} \frac{1}{n} \log |f(\sig^n \om)|=0$).
\end{thm}

The proof of \ref{thm:GranderOseledetsSplitting} occupies the rest of the
section.  First, we present Lemma \ref{lem:goodComplement}, that allows
us to choose complementary spaces in the filtration of Theorem
\ref{thm:Son}. Then, Lemma \ref{thm:OseledetsSplitting} provides an
inductive step that establishes the proof of
Theorem~\ref{thm:GranderOseledetsSplitting}.

\begin{lem}[Existence of a good complement]\label{lem:goodComplement}
Let the filtration $V_1(\omega)\supset\ldots\supset V_{l+1}(\omega)$ be
as in Theorem~\ref{thm:Son}. Then, for every $1\leq j \leq l$,
there exist $m_j$ dimensional spaces $U_j(\om)$ such that the following
conditions hold.
\begin{enumerate}
\item \label{lem:comp}
For $\bbp$-almost every $\om \in \Om$, $V_{j+1}(\om)\oplus
U_j(\om)=V_{j}(\om)$.

\item \label{lem:meascomp}
The map $\om \mapsto U_j(\om)$ is $(\mc{F}, \mc{B}_{\mc{G}}(X))$ measurable.

For $j=1$, let $U_{<j}(\om)=\{0\}$, and for $1<j\leq l$, let
$U_{<j}(\om)=\bigoplus_{i=1}^{j-1} U_i(\om)$. Then,
\item \label{lem:boundedproj}
$ \| \Pi_{U_j\|V_{j+1}\oplus U_{<j}(\cdot)}\|, \|\Pi_{V_{j+1}\|U_j\oplus
U_{<j}(\cdot)}\| \in L^\infty(\Om, \mc{F}, \bbp)$.
\end{enumerate}
\end{lem}

\begin{proof}[Proof of Lemma~\ref{lem:goodComplement}]
We proceed by induction on $j$. Fix some $1\le j\le l$. If $j>1$, assume the statement has been obtained for 
all $1\le j'<j$.  Let $V(\omega)=V_j(\omega)$, $V_+(\omega)=V_{j+1}(\omega)$ and $k=m_j$.
Also let $U_-(\omega)=\{0\}$ if $j=1$ and $U_-(\omega)=\bigoplus_{j'<j}U_{j'}(\omega)$ if $j>1$.
Let $(x_i)_{i\in \mathbb N}$ be a countable dense subset of the unit sphere in $X$.

Let $\epsilon>0$ be a constant to be fixed later in the proof.
For $1\le l\le k$, we claim that we can
 inductively pick measurable families of vectors $u_l(\omega)\in V_j(\omega)$ satisfying
(a) $\|u_l(\omega)\|=1$; (b) $d(u_l(\omega),W_l(\om))>1-\epsilon$, where $W_l(\omega)=U_{<j}(\om)\oplus V_+(\omega)\oplus\text{span}(u_1(\omega),\ldots,u_{l-1}(\omega))$.

Assume $1\le l\le k$ and that $u_1(\omega),\ldots,u_{l-1}(\omega)$ have already been constructed.
Let $r_1(\omega)=\min\{i\in\mathbb N\colon d(x_i,W_l(\omega))>1-\epsilon/2\text{ and }d(x_i,V(\omega))<\epsilon/2\}$.
Then define subsequent terms of a sequence $(r_s(\omega))_{s\ge 1}$ 
by $r_{s}(\omega)=\min\{i\in\mathbb N\colon d(x_i,x_{r_{s-1}(\omega)})<\epsilon/2^s; 
d(x_i,V(\omega))<\epsilon/2^s\}$. Then $(x_{r_s(\omega)})_{s\ge 1}$ is a sequence of measurable functions
pointwise convergent to a measurable function $u_l(s)$ satisfying the required properties.

To check the last condition, we let 
$\Pi_i=\Pi_{\text{span}(u_i(\om))\| W_i(\omega)}$. It follows from the above that $\|\Pi_i\|\leq 1/(1-\ep)$.
Also, $\Pi_{U_j\|V_{j+1}\oplus U_{<j}(\cdot)}$ and $\Pi_{V_{j+1}\|U_j\oplus
U_{<j}(\cdot)}$ can be expressed as a finite sum of compositions of $\Pi_i$ and $I$.
The result follows.

\end{proof}

As before, fix some $1\leq j \leq l$ and let $V(\om)=V_j(\om),
V_+(\om)=V_{j+1}(\om)$, $k=m_j=\text{codim}(V_+, V)$, $\lam=\lam_j$
and $\mu=\lam_{j+1}$.

\begin{lem}\label{thm:OseledetsSplitting}
  Let $\mc{R}=(\Om,\mc{F}, \bbp, \sig, X, \mcl)$ be a separable strongly measurable
  random linear system with ergodic invertible base. Let
  $\log^+\|\mcl(\omega)\|\in L^1(\Omega,\mathcal F,\mathbb P)$ and
  $\mc{R}$ be quasi-compact.  Let $U(\om)$ be a good complement of
  $V_+(\om)$ in $V(\om)$, as provided by
  Lemma~\ref{lem:goodComplement}.  For $n\geq 0$, let $Y^{(n)}(\om)=
  \mcl_{\sig^{-n}\om}^{(n)} U(\sig^{-n}\om)$. Then, for $\bbp$-almost
  every $\om$ the following holds.
\begin{enumerate}
  \item \label{lem:conv} (Convergence)
  As $n\to \infty$, $Y^{(n)}(\om)$ converges to a $k$-dimensional
  space $Y(\om)$, which depends measurably on
  $\om$.
  \item \label{lem:invcomp} (Equivariant complement)
   $V_+(\om)\oplus Y(\om)=V(\om)$.
  Hence, for all $y\in Y(\om)\setminus \{0\}$, $\lim_{n \to \infty}
  \frac{1}{n}\log \|\mcl_{\om}^{(n)}y\|=\lam$.
  Furthermore, $\mcl_\om Y(\om)= Y(\sig \om)$.
  \item \label{lem:uniq} (Uniqueness)
  $Y(\om)$ is independent of the choice of $U(\om)$.

  Let $\tld{Y}(\om)= \bigoplus_{i \leq j}Y_i(\om)$. Then,
\item \label{lem:temp} (Temperedness)
  The norms of projections $\Pi_1(\om):=\Pi_{V_{+}\|\tld{Y}(\om)}$ and
  $\Pi_2(\om):=\Pi_{\tld{Y}\|V_+(\om)}$ are tempered with respect to
  $\sig$.
\end{enumerate}
\end{lem}

Before proceeding to the proof of Lemma~\ref{thm:OseledetsSplitting}, let
us collect some facts that will be used in it. For $j=1$, let
$U_-(\om)=\{0\}$, and for $1<j\leq l$, let
$U_-(\om)=\bigoplus_{i=1}^{j-1} U_i(\om)$. Then, we have that
$\Pi_{V\|U_-(\om)}=\pvu{\om} + \puv{\om}$. Also, by invariance of
$V(\om)$ under $\mcl_\om$, we have that
\[
\mcl_\om \circ \Pi_{V\|U_-(\om)}=(\pvu{\sig\om} + \puv{\sig\om}) \mcl_\om
(\pvu{\om} + \puv{\om}).
\]
Let
\begin{align*}
  \mcl_{00}(\om)&= \pvu{\sig\om} \mcl_\om \pvu{\om},\\
\mcl_{01}(\om)&= \puv{\sig\om} \mcl_\om \pvu{\om}, \\
 \mcl_{10}(\om)&= \pvu{\sig\om} \mcl_\om \puv{\om},\\
\mcl_{11}(\om)&= \puv{\sig\om} \mcl_\om \puv{\om}.
\end{align*}
Note that by invariance of $V_+$, $\mcl_{01}(\om)=0$, $\bbp$-almost
surely. Therefore,  $\mcl_\om \circ \Pi_{V\|U_-(\om)}= \mcl_{00}(\om)+
\mcl_{10}(\om)+ \mcl_{11}(\om)$.

Let $\mcl_{00}^{(n)}(\om)= \pvun{n} \mcl_\om^{(n)} \pvu{\om}$, and define
operators $\mcl_{01}^{(n)}(\om)$ and $\mcl_{11}^{(n)}(\om)$ analogously.
It is straightforward to verify the following identities.
\begin{align}
\mcl_{00}^{(n)}(\om)&= \mcl_{00}(\sig^{n-1}\om)\dots \mcl_{00}(\om),
\label{eq:L00^n} \tag{$L_{00}$}\\
\mcl_{11}^{(n)}(\om)&= \mcl_{11}(\sig^{n-1}\om)\dots \mcl_{11}(\om).
\label{eq:L11^n} \tag{$L_{11}$}
\end{align}
By induction, we also have that
\begin{equation}
\mcl_{10}^{(n)}(\om)=\sum_{i=0}^{n-1} \mcl_{00}^{(i)}(\sig^{n-i}\om)
\mcl_{10}(\sig^{n-i-1}\om) \mcl_{11}^{(n-i-1)}(\om). \label{eq:L10^n}
\tag{$L_{10}$}
\end{equation}

\begin{slem}\ \label{slem:growthrates} Under the assumptions of
  Lemma~\ref{thm:OseledetsSplitting}, the following statements hold.
\begin{enumerate}
\item \label{slem:L00}
For $\bbp$-almost every $\om \in \Om$,
\[
\lim_{n \to \infty} \frac{1}{n}\log \|\mcl_{00}^{(n)}(\om)\|\leq \mu.
\]
Furthermore, for every $\ep>0$ and $\bbp$-almost every $\om \in \Om$,
there exists $D_1(\om)<\infty$ such that for every $i\geq 0$,
\[
\|\mcl_{00}^{(i)}(\sig^{-i}\om)\| \leq D_1(\om) e^{i(\mu+\ep)}.
\]
Applying this with $\omega$ replaced by $\sigma^n\omega$ we obtain
\[
\|\mcl_{00}^{(i)}(\sig^{n-i}\om)\| \leq D_1(\sig^n\om) e^{i(\mu+\ep)}.
\]
\item \label{slem:L10}
For every $\ep>0$ and $\bbp$-almost every $\om \in \Om$, there exists
$D_2(\om)<\infty$ such that for every $n\in \Z$,
\[
\|\mcl_{10}(\sig^{n}\om)\| \leq D_2(\om) e^{|n|\ep}.
\]
\item \label{slem:L11}
 For $\bbp$-almost every $\om \in \Om$, and every $u\in U(\om) \setminus \{0\}$,
\[
\lim_{n \to \infty} \frac{1}{n}\log \|\mcl_{11}^{(n)}(\om)u\|=\lam.
\]
In particular, $\mcl_{11}(\om)|_{U(\om)}: U(\om) \to U(\sig \om)$ is
invertible for $\bbp$-almost every $\om$.  Furthermore, for every $\ep>0$
and $\bbp$-almost every $\om \in \Om$, there exists $C(\om)<\infty$ such
that for every $n\geq 0$, and every $u\in U(\sigma^{-n}\omega)$
satisfying $\|u\|=1$
\[
\|\mcl_{11}^{(n)}(\sig^{-n}\om)u\| \geq C(\om) e^{n(\lam-\ep)}.
\]
\end{enumerate}
\end{slem}

\begin{proof}

  \ \\ \textbf{Proof of \eqref{slem:L00}.}\\
  From the definition, $\| \mcl_{00}^{(n)}(\om) \| \leq \| \pvun{n}\| \|
  \mcl_\om^{(n)}|_{V_+(\om)}\| \|\pvu{\om}\|$. By Theorem \ref{thm:Son}, we have
  that $\lim_{n \to \infty} \frac{1}{n}\log \|
  \mcl_\om^{(n)}|_{V_+(\om)}\|=\mu$. Using that $\pvu{\om}$ is tempered
  with respect to $\sig$, which follows from
  Lemma~\ref{lem:goodComplement}\eqref{lem:boundedproj}, we get that
  for $\bbp$-almost every $\om$,
\[
\lim_{n \to \infty} \frac{1}{n}\log \| \mcl_{00}^{(n)}(\om) \|\leq \mu.
\]
The second claim follows exactly like Claim C in the predecessor
paper, \cite{FLQ1}.\\

\ \\ \textbf{Proof of \eqref{slem:L10}.}\\
$\|\mcl_{10}(\sig^{n}\om)\|\leq \|\pvun{n+1}\| \|\mcl_{\sig^{n}\om}\|
\|\puvn{n}\|$. Since $\log^+ \|\mcl_{\om}\|$ is integrable with respect
to $\bbp$, using the Birkhoff ergodic theorem, one sees that $\lim_{n \to
\pm
  \infty} \frac{1}{|n|}\log \|\mcl_{\sig^{n}\om}\|=0$ for $\bbp$-almost
every $\om$.  Using again that $\puv{\om}$ and $\pvu{\om}$ are tempered
with respect to $\sig$, we get that for $\bbp$-almost every $\om$,
\[
\lim_{n \to \pm \infty} \frac{1}{|n|}\log \|\mcl_{10}(\sig^{n}\om)\|=0.
\]
Thus, the claim follows.\\

\ \\ \textbf{Proof of \eqref{slem:L11}.}\\

We use the bases $\{u_1(\om), \dots, u_k(\om)\}$ and $\{u_1(\sig \om),
\dots, u_k(\sig \om)\}$ constructed in the proof of
Lemma~\ref{lem:goodComplement} to express $\mcl_{11}|_{U(\om)}: U(\om)
\to U(\sig \om)$ in matrix form.  Recall that the norms of $\{u_1(\om),
\dots, u_k(\om)\}$ are bounded functions of $\om$. Also from the proof of
Lemma~\ref{lem:goodComplement}\eqref{lem:comp}, we have that
\[
\Big\|\sum_{i=1}^k a_i u_i(\om)\Big\| \geq \frac{3^{-k}(1-5\ep)}{2}
\max_{1\leq i \leq k}|a_i|.
\]
Condition \eqref{lem:boundedproj} of Lemma~\ref{lem:goodComplement}
implies that the multiplicative ergodic theorem of Oseledets
\cite{Oseledets} applies to $\mcl_{11}$. Hence, convergence of
$\frac{1}{n}\log \|\mcl_{11}^{(n)}(\om)u\|$ follows from
Equation~\eqref{eq:L11^n}. Call this limit $\Lam$. Thus, for every
$\ep>0$ there exists a constant $D(\om, u)<\infty$ such that
$\|\mcl_{11}^{(n)}(\om)u\| \leq D(\om, u) e^{n(\Lam+\ep)}$.

On the one hand, by definition and invariance of $V(\om)$, we have that
\begin{align*}
\mcl_{11}^{(n)}(\om)&= \puvn{n} \mcl_\om^{(n)} \puv{\om}\\
&=\puvn{n} \mcl_\om^{(n)}|_{V(\om)} \puv{\om}.
\end{align*}
Therefore,
\begin{align*}
\Lam&=\lim_{n \to \infty} \frac{1}{n}\log^+ \|\mcl_{11}^{(n)}(\om)u\| \\
&\leq \lim_{n \to \infty} \frac{1}{n}\big( \log^+ \| \puvn{n} \| + \log^+
\|\mcl_{\om}^{(n)}|_{V(\om)}\| + \log^+ \| \puv{\om} \| \big)=\lam.
\end{align*}
The last equality follows from Theorem~\ref{thm:Son} and temperedness of
$\puv{\om}$ with respect to $\sig$.

On the other hand, for any $u\in U(\om) \setminus \{0\}$,
\begin{align*}
\|\mcl_{\om}^{(n)}u\| &=\|\mcl_{10}^{(n)}(\om)u + \mcl_{11}^{(n)}(\om)u\|
\leq
\|\mcl_{10}^{(n)}(\om)u\|+\|\mcl_{11}^{(n)}(\om)u\|\\
&= \Big\|\sum_{i=0}^{n-1} \mcl_{00}^{(i)}(\sig^{n-i}\om)
\mcl_{10}(\sig^{n-i-1}\om) \mcl_{11}^{(n-i-1)}(\om)u
\Big\|+\|\mcl_{11}^{(n)}(\om)u\| \\ &\leq \sum_{i=0}^{n-1}
\big\|\mcl_{00}^{(i)}(\sig^{n-i}\om) \big\|
\big\|\mcl_{10}(\sig^{n-i-1}\om)\big\| \big\|\mcl_{11}^{(n-i-1)}(\om)u
\big\|+\|\mcl_{11}^{(n)}(\om)u\| .
\end{align*}
Let us estimate the first sum. In view of parts \eqref{slem:L00} and
\eqref{slem:L10}, we have that
\begin{align*}
\sum_{i=0}^{n-1} & \big\|\mcl_{00}^{(i)}(\sig^{n-i}\om) \big\|
\big\|\mcl_{10}(\sig^{n-i-1}\om)\big\| \big\|\mcl_{11}^{(n-i-1)}(\om)
u\big\|
\\ &\leq D_1(\sig^n \om) D_2(\om) D(\om, u) \sum_{i=0}^{n-1} e^{i(\mu+\ep)}
e^{(n-i-1)\ep} e^{(\Lam+\ep)(n-i-1)} \\ &\leq D_1(\sig^n \om) D_2(\om) D(\om, u)
\sum_{i=0}^{n-1} e^{(n-1)(\max(\Lam, \mu) +2\ep)}.
\end{align*}
Let $M>0$ be such that $\bbp(D_1(\om)<M)>0$. Then, by ergodicity of
$\sig$, for $\bbp$-almost every $\om$, there are infinitely  many $n$
such that $D_1(\sig^n \om)<M$. For every such $n$ we have that
\[
\|\mcl_{10}^{(n)}(\om)u\| \leq M D_2(\om) D(\om, u) n e^{(n-1)(\max(\Lam,
\mu) +2\ep)}.
\]
Hence,
\[
\liminf_{n\to \infty}\frac{1}{n} \log \|\mcl_{10}^{(n)}(\om)u\| \leq
\max(\Lam, \mu) .
\]
By definition of $\Lam$, we also know that
\[
\lim_{n\to \infty}\frac{1}{n} \log \|\mcl_{11}^{(n)}(\om)u\|=\Lam.
\]
Therefore,
\[
\lam=\lim_{n\to \infty}\frac{1}{n} \log \|\mcl_{\om}^{(n)}u\|\leq
\max(\Lam, \mu).
\]
Recalling that $\mu<\lam$, we get that $\lam \leq \Lam$.
Combining with the argument above, we conclude that $\lam=\Lam$ as claimed.

The almost-everywhere invertibility of $\mathcal
L_{11}(\omega)|_{U(\om)}$ follows immediately.

The last statement follows as in the case of matrices in the
predecessor paper \cite[Lemma 8.3]{FLQ1}.
\end{proof}

The following lemma will be useful in the proof of Lemma
\ref{thm:OseledetsSplitting}\eqref{lem:uniq}.
\begin{lem}\label{lem:uniformGrowth}
Assume $Y'(\om)$ is a measurable equivariant complement of $V_+(\om)$ in
$V(\om)$. From Theorem~\ref{thm:Son}, for every $y'\in Y'(\om)\setminus
\{0\}$,  $\lim_{n \to \infty} \frac{1}{n}\log
\|\mcl_{\om}^{(n)}y'\|=\lam$. Furthermore, for every $\ep>0$ and
$\bbp$-almost every $\om \in \Om$, there exists $C'(\om)>0$ such that for
every $n\geq 0$, and every $y'\in Y'(\omega)$ satisfying $\|y'\|=1$,
\[
\|\mcl_\om^{(n)}y'\| \geq C'(\om) e^{n(\lam-\ep)}.
\]
\end{lem}
\begin{proof}
  The proof follows from the corresponding statement for matrices due
  to Barreira and Silva \cite{BarreiraSilva}, as in \cite[Lemma
  19]{FLQ2}, with the only difference being the choice of suitable
  bases for $Y'(\om)$, which in our setting may be done in a
  measurable way similar to that used for the proof of
  Lemma~\ref{lem:goodComplement}.
\end{proof}

\begin{proof}[Proof of Lemma~\ref{thm:OseledetsSplitting}]
\ \\ \textbf{Proof of \eqref{lem:conv}.}\\
The proof follows closely that presented in \cite[Theorem 4.1]{FLQ1}, for
the case of matrices. First, we define
\[
g_n(\om)=\max_{u\in U(\om),
  \|u\|=1}\frac{\|\mcl_{10}^{(n)}(\sig^{-n}\om)u\|}
{\|\mcl_{11}^{(n)}(\sig^{-n}\om)u\|}.
\]
Then, using the characterizations from \eqref{eq:L11^n} and
\eqref{eq:L10^n} together with invertibility of $\mcl_{11}(\om)$, we have
that
\[
g_n(\om) \leq \sum_{i=0}^{n-1}\frac{\max_{u\in U(\sig^{-i-1}\om),
\|u\|=1} \|\mcl_{00}^{i}(\sig^{-i}\om) \mcl_{10}(\sig^{-i-1}\om) u
\|}{\min_{u\in U(\sig^{-i-1}\om), \|u\|=1}
\|\mcl_{11}^{i+1}(\sig^{-i-1}\om) u\| }.
\]
Let $\ep<\frac{\lam - \mu}{4}$. Using Sublemma~\ref{slem:growthrates}, we
have that for $\bbp$-almost every $\om$, there is a constant
$C'(\om)<\infty$ such that
\[
g_n(\om) \leq C'(\om) \sum_{i=0}^{n-1} e^{(\mu -\lam +3\ep)i}.
\]
Hence, $M(\om):=\sup_{n\in \N} g_n(\om)<\infty$ for $\bbp$-almost every
$\om$.

Next, we show that the sequence of subspaces $Y^{(n)}(\om)=
\mcl_{\sig^{-n}\om}^{(n)} U(\sig^{-n}\om)$ forms a Cauchy sequence in
$\mc{G}_k(X)$.

Let $m>n$.
By homogeneity of the norm, the expression
\[
\max\{\sup_{x\in Y^{(n)}(\om) \cap B} d(x, Y^{(m)}(\om)\cap B),
\sup_{x\in Y^{(m)}(\om) \cap B}d(x, Y^{(n)}(\om)\cap B) \}
\]
coincides with
\[
\max\{\sup_{x\in Y^{(n)}(\om) \cap S(X)} d(x, Y^{(m)}(\om)\cap B),
\sup_{x\in Y^{(m)}(\om) \cap S(X)}d(x, Y^{(n)}(\om)\cap B) \},
\]
where $S(X)$ is the unit sphere in $X$.

First, let $x\in Y^{(n)}(\om) \cap S(X)$. Then $x=
\mcl^{(n)}_{\sig^{-n}\om}u$, with $u\in U(\sig^{-n}\om)$. Since
$\mcl_{11}^{(m-n)}(\sig^{-m}\om)$ is invertible for $\bbp$-almost every
$\om$, there exists $u'\in U(\sig^{-m}\om)$ such that
$\mcl^{(m-n)}_{\sig^{-m}\om}u'=u+v$, for some $v\in V_+(\sig^{-n}\om)$.
Since $v=\mcl_{10}^{(m-n)}(\sig^{-m}\om)u'$ and
$u=\mcl_{11}^{(m-n)}(\sig^{-m}\om)u'$, we have that $\|v\| \leq
M(\sig^{-n}\om)\|u\|$.

Let $y=\mcl^{(m)}_{\sig^{-m}\om}u' \in Y^{(m)}(\om)$. Then, $y=x+
\mcl^{(n)}_{\sig^{-n}\om}(v)$. Also, $\| \mcl^{(n)}_{\sig^{-n}\om}(v)\|
\leq D_1(\om) e^{n(\mu+\ep)} M(\sig^{-n}\om)\|u\|$, with $D_1(\om)$ as in
Sublemma~\ref{slem:growthrates}\eqref{slem:L00}. Using Sublemma
\ref{slem:growthrates}\eqref{slem:L11}, we also have that
\[
1=\|x\|=\| \mcl^{(n)}_{\sig^{-n}\om}u\| \geq C(\om) e^{n(\lam-\ep)}\|u\|.
\]
Letting $K(\om)=\frac{D_1(\om)}{C(\om)}$ and $\al=\lam -\mu-2\ep$, we get
that $d(x, Y^{(m)}(\om)) \leq \|y-x\| = \|\mcl^{(n)}_{\sig^{-n}\om}(v)\|
\leq K(\om) M(\sig^{-n}\om) e^{-\al n}$. By the triangle inequality, $
d(x, Y^{(m)}(\om) \cap B) \leq 2 d(x, Y^{(m)}(\om))$. Therefore,
\begin{equation}\label{eq:distYmYn1}
d(x, Y^{(m)}(\om) \cap B) \leq 2K(\om) M(\sig^{-n}\om) e^{-\al n}.
\end{equation}

Second, let $y\in Y^{(m)}(\om)\cap S(X)$. Then, $y=
\mcl^{(m)}_{\sig^{-m}\om}u'$ for some $u'\in U(\sig^{-m}\om)$. Let
$\mcl^{(m-n)}_{\sig^{-m}\om}u'=u+v$, with $u \in U(\sig^{-n}\om)$ and $v
\in V_{+}(\sig^{-n}\om)$.

Using once again the definition of $M(\om)$ at the beginning of the
proof, we have that $\|v\|\leq M(\sig^{-n}\om)\|u\|$, combined with
Sublemma~\ref{slem:growthrates} we obtain that
\[
\| \mcl^{(n)}_{\sig^{-n}\om}v \| \leq K(\om) M(\sig^{-n}\om)e^{-\al n} \|
\mcl^{(n)}_{\sig^{-n}\om}u\|.
\]

On the other hand, since $\mcl^{(n)}_{\sig^{-n}\om}(u+v)=y$, we have
\[
\|\mcl^{(n)}_{\sig^{-n}\om}u\| \leq \|y\| +
\|\mcl^{(n)}_{\sig^{-n}\om}v\|\leq 1+ K(\om) M(\sig^{-n}\om)e^{-\al n} \|
\mcl^{(n)}_{\sig^{-n}\om}u\|.
\]
Therefore, whenever $K(\om) M(\sig^{-n}\om)e^{-\al n}<1$, we have
\[\|\mcl^{(n)}_{\sig^{-n}\om}u\|\leq \frac{1}{1-K(\om) M(\sig^{-n}\om)e^{-\al
n}}.
\]

Since $\mcl^{(n)}_{\sig^{-n}\om}u\in Y^{(n)}(\om)$, we have
\[d(y,
Y^{(n)}(\om)) \leq
\|y-\mcl^{(n)}_{\sig^{-n}\om}u\|=\|\mcl^{(n)}_{\sig^{-n}\om}v\|\leq
\frac{K(\om) M(\sig^{-n}\om)e^{-\al n}}{1-K(\om) M(\sig^{-n}\om)e^{-\al
n}}.
\]
As before, we use the triangle inequality to conclude that
\begin{equation}\label{eq:distYmYn2}
d(y, Y^{(n)}(\om)\cap B) \leq\frac{2K(\om) M(\sig^{-n}\om)e^{-\al
n}}{1-K(\om) M(\sig^{-n}\om)e^{-\al n}}.
\end{equation}

Combining \eqref{eq:distYmYn1} and \eqref{eq:distYmYn2},
 we get that
\[
d(Y^{(n)}(\om), Y^{(m)}(\om))\leq \frac{2K(\om) M(\sig^{-n}\om)e^{-\al
n}}{1-K(\om) M(\sig^{-n}\om)e^{-\al n}}.
\]
Therefore,
\[
d(Y^{(m')}(\om), Y^{(m)}(\om))\leq \frac{4K(\om) M(\sig^{-n}\om)e^{-\al
n}}{1-K(\om) M(\sig^{-n}\om)e^{-\al n}}\text{ for every $m, m'>n$,}
\]
provided $K(\om) M(\sig^{-n}\om)e^{-\al n}<1$.

Since $M(\om)<\infty$ for $\bbp$-almost every $\om$, there exists an
$A>0$ such that $\bbp(M(\om)<A)>0$. By ergodicity of $\sig$, for
$\bbp$-almost every $\om$, there exist arbitrarily large values of $n$
for which $M(\sig^n\om)<A$, proving that $Y^{(m)}(\om)$ is a Cauchy
sequence. Therefore, it is convergent. Let us call its limit $Y(\om)$.

Measurability of $Y^{(n)}(\om)$ comes from
Corollary~\ref{cor:ImageIsMble}.  Hence, measurability of $Y(\omega)$
follows from the fact that the pointwise limit of Borel-measurable
functions from a measurable space to a metric space is
Borel-measurable.

\noindent \textbf{Proof of \eqref{lem:invcomp}.}\\
By closedness of $\mc{G}_k(X)$, we know that $Y(\om)\in \mc{G}_k(X)$.
Since $Y(\om)$ is the limit of subspaces of $V(\om)$ and $V(\om)$ is
closed, it follows that $Y(\om)\subset V(\om)$. We also have that
$V_+(\om)$ is a $k$-codimensional subspace of $V(\omega)$. Hence, to show
that $Y(\om)\oplus V_+(\om)=V(\om)$, we must show that $Y(\om)\cap
V_+(\om)=\{0\}$. Let $x\in Y^{(n)}(\om)$, with $\|x\|=1$. Then,
$x=\mcl_{\sig^{-n}\om}^{(n)}u'$ for some $u'\in U(\sig^{-n}\om)$. Writing
$x=u+v$, with $u\in U(\om)$, $v\in V_{+}(\om)$, we have that $\|v\|\leq
M(\om)\|u\|$. Thus, $1=\|x\|\leq \|u\|(1+M(\om))$, which yields
$\|u\|\geq \frac{1}{1+M(\om)}$. Since this holds for every $n$, every
$x\in Y(\om)$ with $\|x\|=1$ when decomposed as $x=u+v$ with $u\in
U(\om)$, $v\in V_{+}(\om)$, also satisfies $\|u\|\geq
\frac{1}{1+M(\om)}$. Therefore $Y(\om)\cap V_+(\om)=\{0\}$.

The second statement follows directly from Theorem~\ref{thm:Son}.

To show invariance, we observe that $\mcl_\om
Y^{(n)}(\om)=Y^{(n+1)}(\sig\om)$. Also, by the previous argument combined
with Sublemma~\ref{slem:growthrates} , we have that for every $n\geq 0$,
$(\mcl_\om, Y^{(n)}(\om)) \in G(N,k,0)$ for some $N>0$ (see notation of
Lemma~\ref{lem:ImagePWCont}), and $(\mcl_\om, Y(\om)) \in G(N,k,0)$ as
well. Hence, by Lemma~\ref{lem:ImagePWCont}, $\lim_{n \to
  \infty}\mcl_\om Y^{(n)}(\om) = \mcl_\om Y(\om)$. Thus,
$\mcl_\om(Y(\om))=Y(\sig \om)$.

\ \\ \textbf{Proof of \eqref{lem:uniq}.}\\
Let $Y_-(\om)=\bigoplus_{i<j} Y_i(\om)$.  Assume $Y'(\om)$ is a
measurable equivariant complement of $V_+(\om)$ in $V(\om)$.  We will
show that $Y'(\om)=Y(\om)$ for $\bbp$-almost every $\om \in \Om$.  Let
$R(\om)= \Pi_{V_+ \| Y \oplus Y_-(\om)}$.
Remark~\ref{rmk:ContOfProjSOTTop} and Lemma~\ref{lem:contOfDirectSum}
imply that $R$ is $(\mc{F}, \mc{S})$ measurable. Let
$h(\om)=\|R(\om)|_{Y'(\om)}\|$. Then, $h$ is non-negative and, in view of
Lemma~\ref{lem:normRestriction}, $h$ is $(\mc{S}, \mc{B}(\R))$
measurable.

We claim that for $\bbp$-almost every $\om$, $\lim_{n\to \infty} h(\sig^n
\om)=0$.  We will show this in the next paragraph. Given this, we can
finish the proof as follows.  Let $E_i=\{ \om: h(\om) \leq \frac{1}{i}
\}$.  The claim implies that $\bbp$-almost every $\om$, $\sig^n \om \in
E_i$ for sufficiently large $n$.  Hence, by the Poincar\'e recurrence
theorem, $\bbp(\Om \setminus E_i)=0$ for all $i \in \N$, and therefore,
since $h$ is non-negative, $h(\om)=0$ for $\bbp$-almost every $\om$.
This implies that $Y'(\om) \subset Y(\om) \oplus Y_-(\om)$. On the other
hand, $Y(\om), Y'(\om) \subset V(\om)$, and by part \eqref{lem:invcomp},
$V(\om) \cap Y_-(\om)=\{0\}$. So $Y'(\om) \subset Y(\om)$. Since $Y(\om)$
and $Y'(\om)$ have the same dimension, $Y'(\om) =Y(\om)$ as claimed.

The proof of $\lim_{n\to \infty} h(\sig^n \om)=0$ proceeds as in \cite[\S
3.2]{FLQ2}.  Let $y'\in Y'(\om) \setminus \{0\}$. Since $R(\om) y' \in
V_+(\om)$, Sublemma~\ref{slem:growthrates}\eqref{slem:L00} implies that
for every $\ep>0$, there exists some $D'(\om)<\infty$ such that
$\|\mcl_{\om}^{(n)} R(\om) y' \| \leq D'(\om) e^{n(\mu+\ep)} \|y'\|$. By
Remark~\ref{lem:uniformGrowth}, for every $\ep>0$, there exists some
$C'(\om)>0$ such that $\|\mcl_{\om}^{(n)} y' \| \geq C'(\om)
e^{n(\lam-\ep)}\|y'\|$.  Let $\ep<\frac{\lam - \mu}{4}$. Then,
$\frac{\|\mcl_\om^{(n)} R(\om) y' \|}{\|\mcl_\om^{(n)} y' \|}\leq
\frac{D'(\om)}{C'(\om)} e^{-n(\lam-\mu - 2\ep)}$ for every $n\geq 0$.
Consider the closed sets
\[
D_N=\{ y'\in Y'(\om): \|\mcl_\om^{(n)} R(\om) y' \| \leq N e^{-n(\lam-\mu
- 2\ep)}\|\mcl_\om^{(n)} y' \| \text{ for all } n\in \N\}.
\]
Since $\bigcup_{N\in \N} D_N=Y'(\om)$, the Baire category principle
implies that there exists $N\in \N$, $y' \in Y'(\om)$ and $\del>0$ such
that $\overline{B_\del(y')} \cap Y'(\om) \subset D_N$. By linearity of
$\mcl_\om$, $\overline{B_1(\frac{y'}{\del})} \cap Y'(\om) \subset D_N$.

Let $x \in Y'(\om)$ with $\|x\|=1$. Then, $\|\mcl_\om^{(n)} R(\om)
(\frac{y'}{\del}+x) \| \leq N e^{-n(\lam-\mu - 2\ep)} \|\mcl_\om^{(n)}
(\frac{y'}{\del}+x)\|$ and $\|\mcl_\om^{(n)} R(\om) (\frac{y'}{\del}) \|
\leq N e^{-n(\lam-\mu - 2\ep)} \|\mcl_\om^{(n)} (\frac{y'}{\del})\|$.

By invariance of $V_+(\om), Y(\om)$ and $Y_-(\om)$, we see that $R(\sig^n
\om) \mcl_\om^{(n)}= \mcl_\om^{(n)} R(\om)$. Hence,
\begin{align*}
  \|R(\sig^n \om) \mcl_\om^{(n)}(x)\| &\leq N e^{-n(\lam-\mu - 2\ep)}
  \Big(\|\mcl_\om^{(n)} (\frac{y'}{\del})\|+ \|
  \mcl_\om^{(n)} (\frac{y'}{\del}+x) \| \Big)\\
  &\leq 2 N e^{-n(\lam-\mu - 2\ep)} C''(\om)e^{n(\lam+\ep)}
  \Big(\|\frac{y'}{\del}\|+1 \Big),
\end{align*}
where the existence of such $C''(\om)<\infty$ is guaranteed by
Theorem~\ref{thm:Son}.  Furthermore, using invariance of $Y'(\om)$, we
get that $\mcl_\om^{(n)}(Y'(\om))=Y'(\sig^n \om)$. Thus,
\begin{align*}
  h(\sig^n \om) &\leq \frac{\sup_{x \in Y'(\om) \cap S^1(X)}
    \|R(\sig^n \om) \mcl_\om^{(n)}(x)\|}{\inf_{x \in Y'(\om) \cap
      S^1(X)}
    \|\mcl_\om^{(n)}(x)\| }\\
  & \leq \frac{2 N C''(\om) (\|\frac{y'}{\del}\|+1)}{C'(\om)}
  e^{-n(\lam-\mu - 4\ep)},
\end{align*}
where the last inequality follows from the existence of constants
$C'(\om)$ and $C''(\om)$ as before.  By the choice of $\ep$, we get
that $\lim_{n\to \infty} h(\sig^n \om)=0$ as claimed.

\ \\ \textbf{Proof of \eqref{lem:temp}.}\\
We want to show that for $i=1,2$ and $\bbp$-almost every $\om\in \Om$,
the following holds
\[
 \lim_{n\to \pm \infty} \frac{1}{n} \log \|\Pi_i (\sig^{n}\om)\|=0.
\]
Since the maps $\Pi_i(\om)$ are projections to non-trivial subspaces,
it follows that all the norms involved are at least 1. We will show
upper bounds. In view of Lemma~\ref{lem:TempFwd=TempBwd}, it suffices
to show
\[
 \lim_{n\to \infty} \frac{1}{n} \log \|\Pi_i (\sig^{-n}\om)\|=0.
\]
It suffices to show that for each $j$, $\lim_{n\to \infty} \frac{1}{n}
\log \|\Pi_{Y_j||V_{j+1}(\sigma^{-n}\omega)}\|= 0$, where the map is
defined on $V_j(\sig^{-n}\om)$.

Fix $j$, and suppose $1/n \log
\|\Pi_{Y_j||V_{j+1}(\sigma^{-n}\omega)}\| \not\to 0$ as $n \to
\infty$.  Then there exists a sequence $n_1<n_2<\ldots$ such that
$\log \| \Pi_{Y_j || V_{j+1}(\sigma^{-n_i}\omega)} \| > \epsilon n_i$.

Pick $\delta < \min(\epsilon/2, (\lambda_j-\lambda_{j+1})/2)$. Then, for
$\bbp$-almost every $\om$, there exist $C(\omega)$, $D(\omega)$ and
$E(\omega)$ such that
\begin{align*}
\|\mathcal L^{(n)}(\sigma^{-n}\omega)|_{V_{j+1}}\| &<
C(\omega)e^{(\lambda_{j+1}+\delta)n}\text{ for all $n$;}\\
\|\mathcal L^{(n)}(\sigma^{-n}\omega)(y)\| &>
D(\omega)e^{(\lambda_j-\delta)n}\text{ for all $n$ and $y \in B\cap
Y_j(\sigma^{-n}\omega)$;}\\
\|\mathcal L^{(n)}(\sigma^{-n}\omega)|_{V_j}\| &<
E(\omega)e^{(\lambda_j+\delta)n}\text{ for all $n$.}
\end{align*}

Now by hypothesis there exists a sequence $y_{n_i} + v_{n_i} \in
Y_j(\sigma^{-n_i}\omega) \oplus V_{j+1}(\sigma^{-n_i}\omega)$ with
$\|y_{n_i}\|\approx 1$, $\|v_{n_i}\|\approx 1$, but
$\|y_{n_i}+v_{n_i}\| < e^{-\epsilon n_i}$.  Then, we have
\begin{align*}
  \|\mathcal L^{(n_i)}(\sigma^{-n_i}\omega)(y_{n_i})\| &>
  D(\omega)e^{(\lambda_j-\delta)n_i};\\
  \|\mathcal L^{(n_i)}(\sigma^{-n_i}\omega)(v_{n_i})\| &<
  C(\omega)e^{(\lambda_{j+1}+\delta)n_i};\\
  \|\mathcal L^{(n_i)}(\sigma^{-n_i}\omega)(y_{n_i}+v_{n_i})\| &<
  E(\omega)e^{(\lambda_j+\delta)n_i}\|y_{n_i}+v_{n_i}\|\\
  &<E(\omega)e^{(\lambda_j + \delta -\epsilon) n_i}.
\end{align*}

The triangle inequality gives
\[
D(\omega)e^{(\lambda_j-\delta)n_i} < E(\omega)e^{(\lambda_j+\delta-
\epsilon)n_i} + C(\omega)e^{(\lambda_{j+1}+\delta)n_i}.
\]
Since $\lambda_j-\delta > \max(\lambda_j+\delta-\epsilon
,\lambda_{j+1}+\delta)$ this gives a contradiction for sufficiently large
$n_i$.  Hence, $\lim_{n\to \infty} \frac{1}{n}\log \| \Pi_2 (\sig^{-n}
\om) \|= 0$ for $\bbp$-almost every $\om$, as claimed.

Using that $\Pi_1+\Pi_2=\mathrm{Id}$, the corresponding statement for
$\Pi_1$ follows immediately.
\end{proof}
 % SIMET
\section{Oseledets splittings for random Piecewise Expanding
maps}\label{sec:RandomLY}

In this section, we present an application of the semi-invertible
operator Oseledets theorem, Theorem~\ref{thm:GranderOseledetsSplitting},
to the setting of random piecewise expanding maps. It is worked out in
considerable generality in the one-dimensional setting and in a special
case in higher-dimensions, based on results of Cowieson. The main result
of this section is Theorem~\ref{thm:OseledetsSplitting4LYMaps}.

We first discuss the class of piecewise expanding maps and random
piecewise expanding dynamical systems in \S\ref{subS:LYMaps}. Then,
fractional Sobolev spaces and some of their relevant properties are
briefly reviewed in \S\ref{subS:fractionalSobolev}. In
\S\ref{subS:TransferOperator}, we recall the definition of the transfer
operator of a piecewise expanding map acting on a fractional Sob\-olev
space. Strong measurability is established in the one-dimensional case.
Quasi-compactness is proved in \S\ref{subS:QC}.

\subsection*{Convention}
Throughout this section, $C_\#$ will denote various constants that are
allowed to depend only on parameters $d, p, t$ and $\al$, as well as on a
$C^\infty$ compactly supported function $\eta:\R^d \to [0,1]$, that is
chosen and fixed depending only on the dimension $d$, as appears in
Thomine \cite{Thomine} and Baladi and Gou\"ezel \cite{BaladiGouezel}.

\subsection{Random Piecewise Expanding Dynamical Systems}\label{subS:LYMaps}

\begin{defn}\label{defn:PWExpMap}
 A map $T$ is called a \emph{piecewise expanding $C^{1+\alpha}$
map} of a compact region $X_0\subset \R^d$ if:
\begin{itemize}
\item
There is a finite (ordered) collection of disjoint subsets of $X_0$,
$O^1,\ldots,O^I$, each connected and open in $\R^d$,  whose boundaries are unions of
finitely many compact $C^1$ hypersurfaces with boundary, and whose union
agrees with $X_0$ up to a set of Lebesgue measure 0;
\item For each $1\le i\le I$, $T|_{O_i}$ agrees with a $C^{1+\alpha}$ map
$T_i$ defined on a neighbourhood of $\overline{O_i}$ such that $T_i$ is a
diffeomorphism onto its image.
\item There exists $\mu>1$ such that for all $x\in \overline {O_i}$,
$\|DT_i(x)(v)\|\ge \mu \|v\|$ for all $v\in\R^d$.
\end{itemize}
\end{defn}

We define $\nint^T:=I$ to be the number of branches of $T$. The collection
$\{O^1,\ldots,O^I\}$ is called the \emph{branch partition} of $T$.
 The collection of $C^{1+\alpha}$ expanding maps of $X_0$ will be denoted
$\text{PE}^{1+\alpha}(X_0)$. The collection of $C^{1+\alpha}$ expanding
maps with a particular branch partition $\mathcal P$ will be denoted by
$\text{PE}^{1+\alpha}(X_0;\mathcal P)$. In the special case where $d=1$
and $X_0=[0,1]$, we denote the collection of maps satisfying the above
conditions $\LY^{1+\alpha}$. In this case, the elements of $\mc{P}$ are intervals.

Finally, we define a metric $d_{\PE}$ on $\PE^{1+\alpha}(X_0)$ as
follows. Let $S,T\in\PE^{1+\alpha}(X_0)$. Let the branches for $S$ be
$(O^S_i)_{i=1}^{\nint^S}$ and for $T$ be $(O^T_i)_{i=1}^{\nint^T}$
(recall that a piecewise expanding map is assumed to consist of an
\emph{ordered} collection of domains and maps). If $\nint^T\ne \nint^S$,
or $O^S_i\cap O^T_i=\emptyset$ for some $i$, we define $d_{\PE}(S,T)=1$.
Otherwise we define
\[
d_{\PE}(S,T)=
   \max_i \|(S_i-T_i)\vert_{O_i^S\cap O_i^T}\|_{C^{1+\alpha}} + \max_i\Big | \|S_i \|_{C^{1+\al}}-  \|T_i \|_{C^{1+\al}} \Big| +\max_i
   d_H(O_i^S, O_i^T),
\]
where $d_H$ denotes Hausdorff distance. In the one-dimensional case we
call the metric $d_\LY$. We endow $\PE^{1+\alpha}(X_0)$ with the Borel
$\sigma$-algebra.

\begin{rmk}
There is a definition of distance for Lasota-Yorke maps, related to the
Skorohod metric, that has been previously used in the literature; see for
instance Keller and Liverani. \cite{KellerLiverani99}. That notion of
distance is not adequate for our purposes, because it allows maps to
behave badly in sets of small Lebesgue measure.
\end{rmk}

\begin{defn}\label{defn:RandomLYmap}
A \textbf{random $C^{1+\alpha}$ piecewise expanding dynamical system} on
a domain $X_0$ is given by a tuple $(\Om,\mc{F},\bbp,\sig,T)$ where
$(\Om,\mc{F},\bbp,\sig)$ is a probability preserving
transformation and $T:\Omega\to\PE^{1+\alpha}(X_0)$ satisfying
\begin{itemize}

\item[R1.] (Measurability)
$T:\Om \rightarrow \PE^{1+\al}(X_0)$ given by $\om \mapsto T_\om$ is
a measurable function.
\item[R2.] (Number of branches)
The function $\om \mapsto \nint^{T_\om}$ is $\bbp$-log-integrable,
$\nint^{T_\om}$ being the number of branches of $T_\om$.
\item[R3.] (Distortion)
There exists a constant $\dist$ such that $\| DT_\om \|_{\al} \leq \dist$
for $\bbp$-almost every $\om \in \Om$.
\item[R4.](Minimum expansion)
There exists $a<1$ such that for $\bbp$-almost every $\om \in \Om$, 
$\| \mu_{T_\om}^{-1}\|_{\infty}\leq a$, where
$\mu_{T_\om}(x):=\inf_{\|v\|=1} \|DT_{\om}(x)v\|$.
\item[R5.](Branch geometry)
There exists a constant $L$ such that for $\bbp$-almost every
$\om\in\Om$, each branch domain $O_i$ is bounded by at most $L$ $C^1$
hypersurfaces.
\end{itemize}
\end{defn}

A random piecewise expanding dynamical system will denote a random
$C^{1+\alpha}$ piecewise expanding dynamical system for some $0<\al \leq
1$. In the case where $d=1$ and $X_0=[0,1]$, we refer to these systems as
\textbf{random Lasota-Yorke type dynamical systems}. We will also refer
to random $C^2$ piecewise expanding dynamical systems (with the obvious
definition).

\subsection{Fractional Sobolev spaces}\label{subS:fractionalSobolev}

Here we introduce spaces of functions suitable for our purposes. Their
choice is motivated by recent work of Baladi and Gou\"ezel
\cite{BaladiGouezel}. Much of the development in this subsection
parallels that done in \cite{BaladiGouezel} (see also Thomine's work
\cite{Thomine}, the specialization of \cite{BaladiGouezel} to the
expanding case).

While the other works consider the case of a single map, we work with
random dynamical systems. One new feature is that we need to ensure the
strong measurability of the family of Perron-Frobenius operators. In the
context of a single map, it is often sufficient to prove inequalities
with constants depending on the map, showing only that the constants are
finite. A second new feature in the random context is that one needs to
maintain control of the quantities describing compositions of maps as the
inequalities are iterated.

For this reason we give references to the earlier works where possible
and emphasize those points where differences arise.

Let $t\geq 0$ and $1<p<\infty$. Let $\lhpt(\R^d)$, or simply $\lhpt$, be the image of $L_p(\R^d)$
under the injective linear map $\mc{J}_t: L_p(\R^d) \rightarrow L_p(\R^d)$
given by
\[
\mc{J}_t(g)= \mc{F}^{-1}(a_{-t} \mc{F}(g)),
\]
where $\mc{F}$ denotes the Fourier transform, and $a_t(\zeta):=
(1+|\zeta|^2)^{\frac{t}{2}}$.
$\lhpt$ endowed with the norm
\[
\|f\|_{\lhpt}:=\|\mc{F}^{-1}(a_t \mc{F}(f))\|_{L_p(\R^d)}
\]
is a Banach space known as (local) fractional Sobolev space. Thus,
$\mc{J}_t$ is a surjective isometry from $L_p(\R^d)$ to $\lhpt$. Since
$L_p(\R^d)$ is separable and reflexive, its  isometric image  $\lhpt$ is also separable and reflexive.  Also,
the space of differentiable functions with compact support,
$C_0^\infty(\R^d)$, is dense in $\lhpt$. See for example Strichartz
\cite{Strichartz} and references therein for these and other properties
of $\lhpt$.

This subsection closely follows Baladi and Gou\"ezel. In what follows, we
assume that $p>1$ and $0<t<\min\{\al, \frac{1}{p}\}$. The following
properties will be used in the sequel. The first one is taken from
Triebel \cite[Corollary 4.2.2]{Triebel}, and concerns multiplication by
H\"older functions. The second one is a refinement of a result of
Strichartz \cite[Corollary II3.7]{Strichartz}, and deals with
multiplication by characteristic functions of intervals. The third one deals with multiplication by characteristic functions of higher dimensional sets.
The last one is related to a result of Baladi and Gou\"ezel
\cite[Lemma~25]{BaladiGouezel}, about composition with smooth functions.

\begin{lem}[Multiplication by $C^\al$ functions]\label{lem:multByHolderFunction}
(Triebel \cite[Corollary 4.2.2]{Triebel})\\
There exists a constant $C_\#$, depending only on $t$ and $\al$, such
that for any $g \in C^\al(\R^d,\R)$, and for any $f \in \lhpt$, we have that
$fg\in \lhpt$, with
\[
\| gf \|_{\lhpt} \leq C_\# \|g\|_{C^\al}\|f\|_{\lhpt}.
\]
\end{lem}

\begin{lem}[Multiplication by characteristic functions in one
dimension]\label{lem:multBySmallSupport}\
\begin{itemize}
\item[(a)] (Strichartz \cite[Corollary II3.7]{Strichartz})
There exists some constant $C_\#$ depending only on $t$ and $p$ such that
for every $f\in \lhpt$, and interval $I'\subset \R$,
$\|1_{I'}f\|_{\lhpt}\leq C_\#\|f\|_{\lhpt}$.
\item[(b)]
 Let $f\in \lhpt$. Then, for every $\ep>0$ there exists $\del>0$
 such that whenever $I' \subset \R$ is an interval
 of
length at most $\del$, then $\|1_{I'}f\|_{\lhpt}\leq \ep$.
\end{itemize}
\end{lem}

\begin{proof}[Proof of (b)]
Since $C_0^\infty(\R)$ is dense in $\lhpt$, there exists (using part (a))
$g\in C^\alpha(\R)\cap \lhpt$ such that $\|1_{I'}(f-g)\|_{\lhpt}\le
\ep/2$ for all intervals $I'$. By Lemma \ref{lem:multByHolderFunction} we
have $\|1_{I'}g\|_{\lhpt} \leq C_\# \|g\|_{C^\al}\|1_{I'}\|_{\lhpt}$. If
$I'$ is of length at most $\delta$, then $\|1_{I'}\|_\lhpt\le
C_\#\delta^{1/p-t}$. Hence if $\delta>0$ is chosen sufficiently small we
obtain that for all intervals $I'$ of length at most $\delta$,
$\|1_{I'}g\|_\lhpt\le \ep/2$ completing the proof.
\end{proof}

\begin{lem}[Multiplication by characteristic functions of nice sets]\label{lem:multBySmallSupportHD}\ 
\begin{itemize}
\item[(a)]
(Strichartz \cite[Corollaries II3.7 and II4.2]{Strichartz})
There exists some constant $C_\#$ depending only on $t$ and $p$ such that
for every set $O\subset \R^d$ intersecting almost every line parallel to some coordinate axis in at most $L$ connected components, and for every $f\in \lhpt$, we have that
$\|1_{O}f\|_{\lhpt}\leq C_\# L \|f\|_{\lhpt}$.
\item[(b)]
(Sickel \cite[Proposition 4.8]{Sickel})
Let $\mc{P}$ be a branch partition as in Definition~\ref{defn:PWExpMap}.
Then, there exists a constant $C$ depending on $\mc{P}$ such that for every $O\in \mc{P}$ and every $f$ in $\lhpt$, $1_{O}f \in \lhpt$ with 
$\|1_{O}f\|_{\lhpt}\leq C\|f\|_{\lhpt}$.
\end{itemize}
\end{lem}

\begin{lem}[Composition with $C^{1}$
diffeomorphisms]\label{lem:compWithSmoothFunctions}\
\begin{itemize}
\item[(a)] (Thomine \cite[Lemma 4.3]{Thomine})
Let $A: \R^d \circlearrowleft$ be a linear map. Then, there
exists $C_\#$  such that for any $f\in \lhpt$,
\[
\|f\circ A \|_{\lhpt}\leq C_\# |\det A|^{-\frac{1}{p}}
\|A\|^t\|f\|_{\lhpt} + C_\# |\det A|^{-\frac{1}{p}}\|f\|_{p}.
\]

\item[(b)]
Let $F: \R^d \circlearrowleft$ be a $C^{1}$ diffeomorphism with $\|DF\|_{\infty}, \|DF^{-1}\|_{\infty}<\infty$. Then, there
exists $C_\#$  such that for any $f\in
\lhpt$,
\[
\|f\circ F \|_{\lhpt}\leq C_\#\|\det (DF^{-1})\|_{\infty}^{\frac{1}{p}} \max\{1,
\|DF\|^t_\infty\} \|f\|_{\lhpt}.
\]

\item[(c)]
Let $f\in \lhpt$. Then, for every $\ep>0$ there exists $\del>0$ such that
for every diffeomorphism $F: \R^d \rightarrow \R^d$ with $
\|F-Id\|_{C^{1}}\leq \del$, we have $ \|f\circ F - f\|_{\lhpt} \leq
\ep$.
\end{itemize}
\end{lem}

\begin{proof}
Part (b) follows via interpolation; this result is related to Lemma 4.3 of Thomine \cite{Thomine}.

Now we prove (c). Let $f\in \lhpt$. In view of part (b) and the density
of $C_0^\infty(\R^d)$ in $\lhpt$, we can find a $g\in C_0^\infty$ such that
$\|f-g\|_\lhpt+\|(f-g)\circ F\|_\lhpt<\ep/2$. This gives $\|f\circ
F-f\|_\lhpt\le \ep/2+\|g\circ F-g\|_\lhpt$.

We recall that for $t\leq s$, $H_p^s \subseteq \lhpt$, and the inclusion
is continuous (see for example Strichartz \cite[Corollary
I1.3]{Strichartz}). In particular, for each $t\leq 1$, there exists a
constant $C_\#$ such that for every $g\in H^1_p$, $\|g\|_{\lhpt} \leq
C_\# \|g\|_{H^1_p}$.

Since $g\in C_0^\infty(\R^d)$, $g, g\circ F-g\in H_p^1$. Let $K=\big(1+\|\det(DF^{-1})\|_\infty\big)\leb(\text{supp}(g))<~\infty$. Then
\begin{align*}
  &\|g\circ F-g \|_{\lhpt} \\
  \leq & C_\# \|g-g\circ F \|_{H_p^1} \leq C_\#
  (\|g\circ F-g \|_{L_p} + \|D(g\circ F)-Dg \|_{L_p})
  \\
\leq &C_\# K^{1/p} \Big( \|Dg\|_{\infty}\|I-F\|_{\infty}+ \|D^2
g\|_{\infty}\|I-F\|_{\infty}+ \|Dg\|_{\infty}
  \|1-DF\|_{\infty} \Big),
\end{align*}
since $\leb(\text{supp} h)\leq K$ for all  $h\in \{g\circ F-g, Dg\circ
F-Dg, Dg \circ F\cdot DF-Dg\circ F\}$. Choosing  $\del$ sufficiently
small makes the last expression smaller than $\frac{\ep}{2}$ and hence,
$\|f\circ F - f\|_{\lhpt}\leq \ep$.
\end{proof}

\subsection{Transfer operators}\label{subS:TransferOperator}

Given a map $T\in \PE^{1+\al}(X_0)$ with branches $T_i\colon O_i\to X_0$,
we let $Q_i=T_i(O_i)$ and $\xi_i$ be the inverse branch $T_i^{-1}\colon Q_i \to O_i$.
Assume $p>1$ and $0<t<\min(\alpha,\frac1p)$.

We let $\hpt=\hpt(X_0) \subset \lhpt$ be the subspace of functions supported on the
domain $X_0$, with the induced norm $\|f\|_{\hpt}:=\|f\|_{\lhpt}$. In
view of Lemma~\ref{lem:multBySmallSupportHD}, $\hpt$ is complete, and
thus a Banach space. 
Lemma~\ref{lem:multBySmallSupportHD} shows that $\hpt$ is also separable.
We recall from Baladi \cite[Lemma
2.2]{Baladi-Anisotropic} that the inclusion $\hpt \hookrightarrow
L_p$ is compact.

\begin{rmk}
We note that the space of functions of bounded variation, which has been
the most widely used Banach space to study Lasota-Yorke type maps, is not
separable. This is the reason to look for alternatives. In Baladi
and Gou\"ezel \cite{BaladiGouezel}, the authors show, in particular, that
the fractional Sobolev spaces $\hpt$, are suitable to study transfer
operators associated to piecewise expanding maps.
\end{rmk}

\begin{defn}
The \textit{transfer operator}, $\mcl_T:\hpt \circlearrowleft$,
associated to a map $T\in \PE^{1+\al}(X_0)$ is defined for every $f\in\hpt$ by
\[
  \mcl_T f=\sum_{i=1}^{\nint^T} \big( 1_{O_i^T} \cdot f \big) \circ \xi_i^T \cdot |D\xi_i^T|=  \sum_{i=1}^{\nint^T}  1_{Q_i^T} \cdot f  \circ \xi_i^T \cdot |D\xi_i^T|,
\]
where $|A|$ denotes the absolute value of the determinant of the linear map $A$.
\end{defn}

\begin{rmk}\label{rmk:TransferOpIsOnto}
The results of \S\ref{subS:fractionalSobolev} imply that the linear operators corresponding to composition with smooth functions, multiplication by characteristic functions of elements of the branch partition and multiplication by $C^\al$ functions are bounded in
$\hpt$. Clearly, $\hpt$ is invariant under $\mcl$, and thus the transfer operator acts continuously on $\hpt$.
Furthermore, if $T \in \PE^{1+\al}$ is onto, then $\mcl_T: \hpt \to \hpt$ is
onto. For example, given $f\in \hpt$ and letting $g=\frac{|DT| \cdot
f\circ T}{\sum_{i=1}^{\nint^T}1_{Q_i^T}\circ T}$ gives $\mcl_T g=f$. Again
by the results above, $g\in \hpt$.
\end{rmk}

The following lemma provides a weak continuity property of the transfer
operator acting on a fractional Sobolev space for Lasota-Yorke maps.

\begin{lem}\label{lem:ContTransferOp}
Let $L(\hpt)$ be endowed with the strong operator topology and
$LY^{1+\alpha}$ be endowed with the metric $d_{\LY^{1+\alpha}}$. Then,
the map $\mcl$ sending a Lasota-Yorke map to its transfer operator $\mcl:
\LY^{1+\al} \rightarrow L(\hpt)$ given by $T \mapsto \mcl_T$ is
continuous.
\end{lem}

\begin{proof}
Let $f\in \hpt$ and $T \in \LY^{1+\al}$. We will prove that $\lim_{S
\rightarrow T} \|\mcl_S f - \mcl_T f\|_{\hpt}=0.$ Let $\nint=\nint^T$.
Assume $d_{\LY^{1+\al}}(S,T)<1$. Then, by definition of
$d_{\LY^{1+\al}}$, $\nint^S=b$ . For each $1\leq i\leq \nint $, let
$Q_i^{T \cap S}=Q_i^T \cap Q_i^{S}$ and $Q_i^{T \setminus S}=Q_i^T
\setminus  Q_i^{S}$. Then,
\begin{align*}
  \Big\|\mcl_{T}f - \mcl_{S}f\Big\|_{\hpt} &\leq \sum_{i=1}^{\nint}
  \Big\| 1_{Q_i^{T \cap S}} ( f\circ \xi_i^{T} \cdot |D\xi_i^{T}|
   - f\circ \xi_i^{S} \cdot |D\xi_i^{S}|) \Big\|_{\hpt} \\
&+ \sum_{i=1}^{\nint}\Big\|1_{Q_i^{T \setminus S}}\cdot f\circ \xi_i^{T}
\cdot |D\xi_i^{T}|\Big\|_{\hpt}+
\sum_{i=1}^{\nint}\Big\|1_{Q_i^{S\setminus T}}
 \cdot f\circ \xi_i^{S} \cdot |D\xi_i^{S}|\Big\|_{\hpt}.
\end{align*}
We finish the proof by bounding the terms separately in the following
lemma.
\begin{slem}\label{lem:WeakCont} \
\begin{itemize}
  \item[(I)] Bound on common branches.
  For every $1\leq i \leq \nint$,
\[
  \lim_{S\rightarrow T} \Big\| 1_{Q_i^{T\cap S}} ( f\circ \xi_i^{T} \cdot |D\xi_i^{T}|
   - f\circ \xi_i^{S} \cdot |D\xi_i^{S}|) \Big\|_{\hpt}=0.
\]

\item[(II)] Bound on remaining terms.
For every $1\leq i \leq \nint$,
\begin{align*}
  \lim_{S\rightarrow T} \Big\|1_{Q_i^{T \setminus S}}\cdot  f\circ \xi_i^{T}
  \cdot |D\xi_i^{T}|\Big\|_{\hpt}=0 \quad\text{and}\quad
  \lim_{S\rightarrow T} \Big\|1_{Q_i^{S \setminus T}}\cdot
  f\circ \xi_i^{S} \cdot |D\xi_i^{S}|\Big\|_{\hpt}=0.
\end{align*}

\end{itemize}
\end{slem}
\end{proof}

\begin{proof}[Proof of Sublemma~\ref{lem:WeakCont} (I)]
We start by noting that we can fix a way of choosing extensions of each $T_i$ to a diffeomophism $\tld{T}_i$ of $\R$, in such a way that 
$\| \tld{S}_i- \tld{T}_i\|_{C^{1+\alpha}}\leq 2 \|(S_i-T_i)\vert_{Q_i^S\cap Q_i^T}\|_{C^{1+\alpha}}$. In what follows, we drop the tildes for convenience.
Using Lemmas~\ref{lem:multByHolderFunction}, \ref{lem:multBySmallSupport}
and \ref{lem:compWithSmoothFunctions} repeatedly, we have

\begin{align*}
   &\Big\|1_{Q_i^{T\cap S}} (f\circ \xi_i^{T} \cdot |D\xi_i^{T}| - f\circ \xi_i^{S}
   \cdot |D\xi_i^{S}|)\Big\|_{\hpt} \leq C_\#\Big\|f\circ \xi_i^{T}
    \cdot |D\xi_i^{T}| - f\circ \xi_i^{S} \cdot |D\xi_i^{S}|\Big\|_{\lhpt} \\
&\quad \leq C_\#\Big\| |D\xi_i^{T}|-|D\xi_i^{S}| \Big\|_{C^\al} \Big\|
f\circ \xi_i^{T}\Big\|_{\lhpt}+
C_\#\Big\|D\xi_i^{S}\Big\|_{C^\al} \Big\|f\circ \xi_i^{T}-
f\circ \xi_i^{S}\Big\|_{\lhpt} \\
&\quad \leq  C_\# \Big( \Big\| |D\xi_i^{T}|-|D\xi_i^{S}|
 \Big\|_{C^\al} \Big\|f\Big\|_{\hpt} \Big\|D T_i \Big\|_\infty^{\frac{1}{p}}
  + \Big( \Big\|D\xi_i^{T}\Big\|_{C^\al}+1 \Big)
  \Big\|f\circ \xi_i^{T}-f\circ \xi_i^{S}\Big\|_{\lhpt} \Big),
\end{align*}
where $T_i:= (\xi_i^T)^{-1}$, and in the last inequality we use the fact
that $\big\|D\xi_i^{S}\big\|_{C^\al}< \big \|D\xi_i^{T}\big\|_{C^\al}+1$
whenever $d_{\LY^{1+\al}}(S,T)<1$. The first term goes to 0 as
$S\rightarrow T$ because $d_{\LY^{1+\al}}(S,T)\geq \frac12\big\|
|D\xi_i^S|-|D\xi_i^T| \big\|_{C^\al}.$ It remains to show that
$\lim_{S\rightarrow T} \Big\|f\circ \xi_i^{T}-f\circ
\xi_i^{S}\Big\|_{\lhpt}=0.$ By
Lemma~\ref{lem:compWithSmoothFunctions}(b), showing the above is
equivalent to proving that $\lim_{S\rightarrow T} \Big\|f-f\circ
\xi_i^{S}\circ T_i \Big\|_{\lhpt}=0.$ This is a direct consequence of
Lemma~\ref{lem:compWithSmoothFunctions}(c) and the observation that
$\lim_{S\rightarrow T} \Big\|\xi_i^{S}\circ T_i -
Id\Big\|_{C^{1}}=0$.
\end{proof}

\begin{proof}[Proof of Sublemma~\ref{lem:WeakCont} (II)]
Fix $1\leq i \leq \nint$. First, we observe that 
$\leb(Q_i^{S\setminus T})\leq 2 d_H(Q_i^S, Q_i^T)\leq 2(\|T\|_\infty+1)d_{LY^{1+\al}}(S,T)$.
Since $\xi_i^S$ and $\xi_i^T$ are contracting, then
\begin{equation}\label{eq:SmallLebMeasure}
\lim_{S\rightarrow T}\leb \big(\xi_i^T (Q_{i}^{T \setminus S}) \big)=0
\quad \text{ and } \quad \lim_{S\rightarrow T}\leb \big(\xi_i^S (Q_{i}^{S
\setminus T})\big)=0.
\end{equation}
We now show that
\[
\lim_{S\rightarrow T} \Big\|1_{Q_i^{T \setminus S}}\cdot f\circ \xi_i^{T}
\cdot |D\xi_i^{T}|\Big\|_{\hpt}=0.
\]
Being the set difference of two intervals, the set $Q_i^{T \setminus S}$
is either empty, or an interval, or the union of two intervals. Thus, we
let $Q_i^{T \setminus S}= \bigcup_{\ga_i\in \Gamma_i}Q_{\ga_i}^{T
\setminus S}$ be the decomposition of $Q_i^{T \setminus S}$ into
intervals, where $\# \Gamma_i \in\{0,1,2\}$.

Let  $\ga_i\in \Gamma_i$. Using Lemmas~\ref{lem:multByHolderFunction} and
\ref{lem:compWithSmoothFunctions}(b), respectively, we obtain
\begin{align*}
\Big\|1_{Q_{\ga_i}^{T \setminus  S}} \cdot f\circ \xi_i^{T} \cdot
 |D\xi_i^{T}|\Big\|_{\hpt}
& \leq C_\# \Big\| D \xi_i^T \Big\|_{C^\al}
\Big\|1_{Q_{\ga_i}^{T \setminus  S}} \cdot f\circ \xi_i^{T}\Big\|_{\hpt}\\
& \leq
C_\# \Big\| D \xi_i^T \Big\|_{C^\al} \Big\| D T_i \Big\|_{\infty}^{\frac{1}{p}}
 \Big\|1_{\xi_i^T (Q_{\ga_i}^{T \setminus  S})} f \Big\|_{\hpt}.
\end{align*}
Since $Q_{\ga_i}^{T \setminus S} \subseteq Q_{i}^{T \setminus S}$,
\eqref{eq:SmallLebMeasure} implies that $\lim_{S\to T} \leb( \xi_i^T
(Q_{\ga_i}^{T \setminus S}))=0$. Lemma~\ref{lem:multBySmallSupport}(b)
yields $\lim_{S\rightarrow T} \Big\|1_{\xi_i^T (Q_{\ga_i}^{T \setminus
S})} f \Big\|_{\hpt}=0.$ Therefore,
\[
  \lim_{S\rightarrow T} \Big\|1_{Q_{\ga_i}^{T \setminus S}}
  \cdot f\circ \xi_i^{T} \cdot
  |D\xi_i^{T}|\Big\|_{\hpt}=0,
\]
as claimed. Although the statements of Lemma~\ref{lem:WeakCont}(II) are
not symmetric in $T$ and $S$, interchanging the roles of $T$ and $S$ in
the proof just presented, and recalling from \eqref{eq:SmallLebMeasure}
that $\lim_{S \to T} \leb(\xi_i^{S} (Q_{\ga_i}^{S \setminus T}))=0$, we
get that
\[
  \lim_{S\rightarrow T} \Big\|1_{Q^{S \setminus T}}\cdot f\circ \xi_i^{S}
   \cdot |D\xi_i^{S}|\Big\|_{\hpt}=0.
\]
\end{proof}

Let $\mathcal T=(\Om,\mc{F},\bbp,\sigma,T)$ be a random piecewise
expanding $C^{1+\alpha}$ dynamical system. Suppose that the following
conditions are satisfied.

\begin{itemize}
\item[S1.] (Parameters)
$\hpt$ is the fractional Sobolev space defined in
\S\ref{subS:fractionalSobolev}, with $p>1$ and $0<t<\min\{\al,
\frac{1}{p}\}$.
\item[S2.] (Strong measurability)
The map $\mathcal L$ sending $\omega$ to the transfer operator of
$T_\omega$, $\mcl_{T_\omega}$, acting on $\hpt$, is strongly measurable.
\end{itemize}

Then we call the tuple $(\Om,\mc{F},\bbp,\sigma,\hpt,\mcl)$ the strongly
measurable random linear system \emph{associated} to $\mathcal T$. For
brevity, we will use write $\mcl_\om$ instead of $\mcl_{T_\om}$. Also,
the notation from \S\ref{subS:LYMaps} will be abbreviated. For example,
instead of $b^{T_\om}$, we will write $b^\om$, and so on.

We note that Lemma \ref{lem:ContTransferOp} guarantees that for any
random $C^{1+\alpha}$ Lasota-Yorke dynamical system, condition [S2] is
automatically satisfied. A second situation that we consider is that of a
random piecewise expanding dynamical system in higher dimensions with a
fixed, suitably regular branch partition. In this situation one sees directly that [S2] holds
also.

\subsection{Random Lasota-Yorke Inequalities}

Given a collection $\mathcal C$ of subsets of a set, we recall that its
intersection multiplicity is given by $\max_{x\in\bigcup \mathcal
C}\#\{C\in\mathcal C\colon x\in C\}$. Given $T\in \PE^{1+\al}(X_0)$, the
\emph{complexity of $T$ at the end}, denoted by $C_e(T)$, is the
intersection multiplicity of $\{ \overline{T(O_i^T)} \}_{1\leq i \leq
\nint^T}$. The \emph{complexity of $T$ at the beginning}, $C_b(T)$, is
the intersection multiplicity of $\{\overline{O_i^T} \}_{1\leq i \leq
\nint^T}$. We note that in the one-dimensional Lasota-Yorke case,
$C_b(T)$ is always equal to 2 (even when compositions of maps are taken),
whereas in higher dimensions the complexity at the beginning can grow
without bound as maps are composed. Examples of Tsujii \cite{Tsujii} and
Buzzi \cite{Buzzi:noacim} show that this can lead to singular ergodic
properties of the map including non-existence of absolutely continuous
invariant measures.

As is well-known, quasi-compactness can be derived from Lasota-Yorke type
inequalities of the form $\tn\mathcal Lf\tn\le A\| f\|+B\tn f\tn$, where
$\tn\cdot\tn$ is a stronger norm than $\|\cdot\|$ and the inclusion $(Y, \tn \cdot \tn)\hookrightarrow (Y, \| \cdot  \|)$ is compact. Hennion's theorem
shows that the essential spectral radius is governed by $B$.

The following Lasota-Yorke type inequality is based on results of Thomine
\cite[Theorem 2.3]{Thomine}. In that work, rather than a random dynamical
system, a single dynamical system is considered. Thomine (and the
previous work of Baladi and Gou\"ezel) took a great deal of care to bound
the `$B$' term, but did not need to control the `$A$' term other than to
say that it is finite. In our context, we need the additional fact that
$A$ depends in a measurable way on our dynamical system. That this holds
can be seen by a careful examination of the proofs of Thomine; and Baladi
and Gou\"ezel. One feature of the proof that needs attention is that
these papers replace the norm $\|\cdot\|_\hpt$ by an equivalent norm
depending on properties of the map $T^n$. In our context, we would obtain
results in different norms for different compositions $T^{(n)}_\omega$.
We avoid this at the expense of increasing the $A$ term. More
specifically we make use of a bound of the form
\begin{equation*}
\sum_{m\in\Z}\|\eta_{m,r}u\|^p_\hpt\le
C_\#\left((1+r^{pt})\|u\|_p^p+\|u\|_\hpt^p\right),
\end{equation*}
where $(\eta_{m,r})_{m\in\Z^d}$ is a partition of unity of $\mathbb R^d$
obtained by scaling a fixed partition of unity by a factor $r$ in the
variable, and satisfies $\eta_{m,r}(x)=\eta_{0,r}(x+m/r)$.

\begin{lem}[Strong $L_p-\hpt$ Lasota-Yorke
inequality]\label{lem:LYLongBranches}\ \\
Suppose $\mc{R}=(\Om,\mc{F}, \bbp, \sig, \hpt,\mcl)$ is a strongly
measurable random linear system associated to a random $C^{1+\al}$
piecewise expanding dynamical system $\mc{T}$. Then there exists a
constant $C_{\mc{R}}$, depending only on $p$, $t$, $\alpha, \eta$ and $L$ (from
Definition \ref{defn:RandomLYmap}), and a measurable function $A_{\mc{R},
n}(\om)$ such that for every $\om\in \Om$, we have
\begin{equation}\tag{S-LY}\label{eq:StrongLY}
  \|\mcl_\om^{(n)}f\|_{\hpt}
  \leq  A_{\mc{R}, n}(\om) \|f\|_p  +
  B_{\mc{R},n}(\om)\|f\|_{\hpt},
\end{equation}
where
\begin{align*}
  &A_{\mc{R},n}(\om)\text{ is a measurable function of $\omega$}; and\\
  &B_{\mc{R},n}(\om)=C_{\mc{R}}n\left(C_b(T^{(n)}_\omega)\right)^{\frac1p}
  \left(C_e(T^{(n)}_\omega)\right)^{1-\frac1p}
  \||DT_\omega^{(n)}|^{\frac1p-1}\mu_{\om,n}^{-t}\|_\infty,
\end{align*}
where $\mu_{\om,n}(x):=\inf_{\|v\|=1} \|DT_{\om}^{(n)}(x)v\|$.
\end{lem}

This inequality will prove sufficient to control the index of
compactness, but does not give enough information to control the maximal
Lyapunov exponent since we have no control of the $A_{\mc{R},n}(\omega)$
term. The following inequality remedies the situation by providing an
inequality with no $A$ term, at the expense of having a larger (but still
log-integrable) $B$ term. The availability of both inequalities will
allow us to apply Lemma~\ref{lem:LY-IC-MLE}. The proof of the weak
Lasota-Yorke inequality is straightforward using some of the ingredients
of the stronger version.

\begin{lem}[Weak $L_p-\hpt$ Lasota-Yorke inequality]\label{lem:easyLY}\
\\
Let $\mc{R}=(\Om,\mc{F}, \bbp, \sig, \hpt, \mcl)$ be a random
Lasota-Yorke type dynamical system. Then, for each $n\in \N$ there exists
a $\bbp$-log-integrable function $\tld{A}_{\mc{R}, n}:\Om \to \R$ such
that for every $\om\in \Om$,
\begin{equation}\label{eq:easyLY} \tag{W-LY}
  \|\mcl_\om^{(n)}f\|_{\hpt} \leq  \tld{A}_{\mc{R}, n}(\om)  \|f\|_{\hpt}.
\end{equation}
\end{lem}

\subsection{Quasi-compactness}\label{subS:QC}

In this subsection we prove quasi-compactness by bootstrapping a strategy
of Buzzi \cite{Buzzi} based on multiple Lasota-Yorke inequalities as
elaborated in Appendix~\ref{sec:randomHennion}.

\begin{lem}\label{lem:essinfDT}
Let $\mc{R}=(\Om,\mc{F}, \bbp, \sig, T)$ be a random piecewise expanding
dynamical system with ergodic base. Then the following hold:
\begin{enumerate}
\item There exist $C_e^*<\infty$ and $C_b^*<\infty$ such that
for $\bbp$-almost every $\om \in \Om$,
\begin{align*}
 \lim_{n\rightarrow
 \infty}(C_e(T_{\om}^{(n)}))^{\frac{1}{n}}=C_e^*;\text{\quad
 and\quad}
 \lim_{n\rightarrow
 \infty}(C_b(T_{\om}^{(n)}))^{\frac{1}{n}}=C_b^*;
\end{align*}

\item There exists $\chi<1$ such that
for $\bbp$-almost every $\om \in \Om$,
$
 \lim_{n\rightarrow \infty}\|\mu_{\om,n}^{-1}\|_{\infty}^{\frac{1}{n}}=\chi.
$
Furthermore, 
$
 \lim_{n\rightarrow \infty}\| |DT_{\om}^{(n)}|^{-1}\|_{\infty}^{\frac{1}{n}}\leq \chi^d.
$
\end{enumerate}
\end{lem}

\begin{proof}
The sequences $(C_e(T_\om^{(n)}))_{n\in \N}$, $(C_b(T_\om^{(n)}))_{n\in \N}$,  $( \|\mu_{\om,n}^{-1}\|_{\infty} )_{n\in \N}$ and 
 $( \| |DT_{\om}^{(n)}|^{-1}\|_{\infty})_{n\in \N}$, are
submultiplicative. Log-integrability of $C_e$ and $C_b$ is assured by
Definition \ref{defn:RandomLYmap} since we have
$C_b(T_\om),C_e(T_\om)\leq \log \nint^\om$. Hence the existence of the
limits follows from the Kingman subadditive ergodic theorem
\cite{Kingman}. That $\chi<1$ follows from condition R4.
The last statement follows from $|DT_{\om}^{(n)}(x)|^{-1} \leq \mu_{\om,n}(x)^{-d}$.
\end{proof}

\begin{lem}[Quasi-compactness: Lasota-Yorke case]\label{lem:LYquasiCompactness}
Let $0<\alpha<1$ and let $\mc{T}$ be a $C^{1+\alpha}$ random Lasota-Yorke
dynamical system satisfying the additional condition that the function
$\om \mapsto \log^+ \var(|DT_{\om}|^{-1})$ is $\bbp$-integrable. Then
there exist parameters $p>1, 0<t<\min(\al, \frac{1}{p})$ such that the
associated random linear system, $\mc{R}$, is quasi-compact with
\[
 \ka^*\leq \Big(1-\frac{1}{p} \Big) (\log C_e^*+ \log \chi) + t \log \chi <
 \lam^*=0.
\]
\end{lem}

\begin{proof}[Proof of Lemma~\ref{lem:LYquasiCompactness}] \
Since we are in the Lasota-Yorke case, strong measurability of
$\omega\mapsto\mathcal L_\omega$ follows from Lemma
\ref{lem:ContTransferOp}. We also have $C_b^*=1$. By Lemma
\ref{lem:essinfDT}, $C_e^*<\infty$. By hypothesis, $\chi<1$. Fix
$0<t<\alpha$. Now if $p$ is sufficiently close to 1, $t$ satisfies
$t<\min(\alpha,1/p)$ and the inequality $\Big(1-\frac{1}{p} \Big) (\log
C_e^*+ \log \chi) + t \log \chi < 0 $ holds. By Lemma
\ref{lem:LY-IC-MLE}, we see $\kappa^*<0$. On the other hand we have
$\|\mathcal L_\om^{(n)}1\|_\lhpt\ge C_\#\|\mathcal L_\om^{(n)}1\|_p\ge
C_\#\|\mcl_\om^{(n)}1\|_1=C_\#$ so that $\lambda^*\ge 0$. Accordingly Theorem
\ref{thm:GranderOseledetsSplitting} applies. Suppose for a contradiction
that $\lambda^*>0$.

Now the following are full measure sets:  the set where the results of
Theorem \ref{thm:GranderOseledetsSplitting} hold; the set where the top
Lyapunov exponent is $\lambda^*$; the set where the results of Lemma
\ref{lem:LY-IC-MLE} hold using $\|\cdot\|=\|\cdot\|_1$ and
$\tn\cdot\tn=\|\cdot\|_\text{BV}$ (by Buzzi's argument \cite{Buzzi}); and
the set where the results of Lemma \ref{lem:LY-IC-MLE} hold using
$\|\cdot\|=\|\cdot\|_p$ and $\tn\cdot\tn=\|\cdot\|_\lhpt$, because of Lemmas~\ref{lem:LYLongBranches} and \ref{lem:easyLY}. Let the full
measure set obtained by intersecting these be denoted by $\Omega_1$.

Suppose $\omega\in\Omega_1$ and let $f$ be a non-zero element of
$Y_1(\omega)$. By standard properties of $\lhpt$, $f$ may be approximated
arbitrary closely in $\|\cdot\|_\lhpt$ be a $C^\infty$ function $g$.
Applying Lemma \ref{lem:LY-IC-MLE} with $\|\cdot\|_1$ and
$\|\cdot\|_\text{BV}$, (note that $\|\mcl^{(n)}_\om g\|_1\le \|g\|_1$ for
all $n$), we get $\limsup_{n\to\infty}\frac1n\log\|\mcl^{(n)}_\om
g\|_\text{BV}\le 0$ and hence $\limsup_{n\to\infty}\frac1n\log
\|\mcl^{(n)}_\om g\|_p\le 0$. Applying Lemma \ref{lem:LY-IC-MLE} a second
time using the conclusion of the first application as hypothesis, we
get $\limsup_{n\to\infty}\frac1n\log \|\mcl^{(n)}_\om g\|_\lhpt\le 0$.
Letting $\pi_1$ be the projection onto the top Lyapunov subspace, we have that
$\limsup_{n\to\infty}\frac1n\log\|\mcl^{(n)}_\om
(g-\pi_1(g))\|_\lhpt<\lambda^*$. Thus
$\limsup_{n\to\infty}\frac1n\log\|\mcl^{(n)}_\om \pi_1(g)\|_\lhpt<\lambda^*$.
By Theorem \ref{thm:GranderOseledetsSplitting}, this implies $\pi_1(g)$
is 0. Since $g$ can be chosen arbitrarily close to $f$ and $\pi_1$ is
bounded, this is a contradiction.
\end{proof}

We now show that results of Cowieson \cite{Cowieson} may be exploited to
give families of random dynamical systems in higher dimensions for which
one can establish an Oseledets splitting for the Perron-Frobenius
cocycle. The framework of \cite{Cowieson} has a key simplifying feature,
namely that there is a \emph{fixed} partition $\mathcal P$ of the domain
$X_0$ into disjoint open pieces on each of which the map is continuous
and expanding. For this reason, the analogue of Lemma \ref{lem:WeakCont}
is straightforward: there is no issue with accounting for differences
between partitions. On the other hand a new difficulty appears in higher
dimensions, namely that it is no longer true \textsl{a priori} that the
complexity at the beginning, $C_b(T_\omega^{(n)})$, is bounded in $n$.
The necessity of controlling $C_b$ is demonstrated by results of Tsujii
\cite{Tsujii} and Buzzi \cite{Buzzi:noacim} and indeed $C_b$ appears in
Baladi and Gou\"ezel \cite{BaladiGouezel}.

\begin{thm}[Cowieson\cite{Cowieson}]\label{thm:Cowieson}
Let $\mathcal P$ be a fixed branch partition of a compact region
$X_0\subset \R^d$. There is a quantity $M$ and a dense $G_\delta$ subset,
$\text{Cow}(X_0;\mathcal P)$, of $\PE^{2}(X_0;\mathcal P)$ with the
following property: For any $n>0$ and $T\in \text{Cow}(X_0;\mathcal P)$,
there is a neighbourhood $U$ of $T$ such that for any
$T_1,T_2,\ldots,T_n\in U$, $C_b(T_n\circ\cdots\circ T_1)\le M$.
\end{thm}

\begin{lem}[Quasi-compactness: Cowieson case]
\label{lem:CowquasiCompactness} Let $d>1$, let $X_0$ be a compact region
of $\R^d$ and let $\mathcal P$ be a branch partition of $X_0$. Let
$T\in\text{Cow}(X_0;\mathcal P)$. Then there exist parameters $p>1$,
$0<t<1/p$, a constant $\tau<0$ and a neighbourhood $N$ of $T$ with the
following property:

Let $\mc{T}$ be a  random $C^{2}$ piecewise expanding dynamical system
with ergodic base. Suppose that for $\mathbb P$-almost every $\omega$,
$T_\omega$ has branch partition $\mathcal P$ and that $T_\omega\in N$.
Then if $\mathcal L_\omega$ is the corresponding family of transfer
operators acting on $\hpt(X_0)$, then $\mc{R}$, the random linear system associated to $\mc{T}$, is quasi-compact with
\[
\ka^*\le \tau<\lambda^*=0.
\]
\end{lem}

\begin{proof}
  Let $t=\frac12$, let $T\in\text{Cow}(X_0;\mathcal P)$, and let $k$ be the
  number of elements of $\mathcal P$. Let $a<1$ be such that
  $\mu_T^{-1}< a$, where $\mu_T=\essinf_{x\in X_0, \|v\|=1}\|DT(x)v\|$, and let $M$ be as in guaranteed
  by Theorem \ref{thm:Cowieson}. Let $p>1$ be such that
  $k^{1-\frac1p}a^{d(1-\frac1p)+t}<1$. We may further assume that $p<d/(d-1)$.
  This fixes all the data necessary to determine $C_\mathcal R$.

  Let $n_0$ be such that $\beta:=C_\mathcal
  Rn_0M^{1/p}k^{n_0(1-\frac1p)}a^{n_0(d(1-\frac1p)+t)}<1$ and let
  $\tau=\log\beta/n_0$ so that $\tau<0$. We now apply
  Theorem \ref{thm:Cowieson} to deduce that there is a neighbourhood $N$
  of $T$ such that for all any $n_0$-fold composition of elements of $N$,
  each respecting the branch partition $\mathcal P$,
  the complexity at the beginning is bounded above by $M$. We further
  reduce $N$ (to a smaller open neighbourhood of $T$)
  by requiring that $\mu_S^{-1}<a$ for all $S\in N$.

  Now if $\mathcal{T}$ is a random $C^2$ piecewise expanding dynamical
  system where the maps all belong to $N$ then we have ensured that the
  quantity $B_{\mathcal R,n_0}$ appearing in Lemma \ref{lem:LYLongBranches} is
  at most $e^{n_0\tau}$. As in the proof of Lemma
  \ref{lem:LYquasiCompactness} we obtain $\kappa^*\le \tau$.

  To see that $\lambda^*=0$, we argue as in Lemma
  \ref{lem:LYquasiCompactness}. We initially apply Lemma
  \ref{lem:LY-IC-MLE} with $\|\cdot\|=\|\cdot\|_{L^1}$ and
$\tn\cdot\tn=\|\cdot\|_{BV}$ to deduce for $f\in C^\infty$
  (using results from Cowieson's
  paper \cite{Cowieson}) that $\limsup_{n \to \infty}\frac1n\log\|\mcl^{(n)}_\om f\|_\text{BV}\leq 0$. Then, since there are constants $C_\#, C_\#(X_0)$ such that for any function $g$ supported on $X_0$, $\|g\|_{L^{d/(d-1)}} \le C_\#\|g\|_\text{BV}$  (see Giusti
  \cite[Theorem  1.28]{Giusti}) and $\|g\|_{L^{p}} \leq C_\#(X_0) \| g \|_{L^{d/(d-1)}}$, we obtain sufficient conditions
  for the second
  application of Lemma \ref{lem:LY-IC-MLE}, taking this time
  $\|\cdot\|=\|\cdot\|_{L^{p}}$ and $\tn\cdot\tn=\|\cdot\|_\hpt$.
  The remainder of the
  proof is exactly as in Lemma \ref{lem:LYquasiCompactness}.
\end{proof}

In view of the quasi-compactness just obtained, we can apply
Theorem~\ref{thm:GranderOseledetsSplitting} to get our main application
theorem, ensuring the existence of an Oseledets splitting for random
Lasota-Yorke dynamical systems or Cowieson-type random piecewise
expanding dynamical systems.

\begin{thm}\label{thm:OseledetsSplitting4LYMaps}
Let $\mc{R}=(\Om,\mc{F}, \bbp, \sig, T)$ be a random $C^{1+\al}$
piecewise expanding dynamical system satisfying the hypotheses of
Lemma~\ref{lem:LYquasiCompactness} or Lemma \ref{lem:CowquasiCompactness}
with parameters $p$ and $t$. Then, there exist $1\le l\le\infty$, and
exceptional Lyapunov exponents $0=\lam_1>\lam_2>\dots > \lam_l>\ka^*$ (
in the case $l=\infty$, we have $\lim \lam_n=\ka^*$), measurable families
of finite-dimensional equivariant spaces $Y_1(\om), \dots,
Y_l(\om)\subset X$ and a measurable equivariant family of closed
subspaces $V(\om)\subset X$ defined on a full $\bbp$ measure,
$\sig$-invariant subset of $\Om$ so that $X=V(\om) \oplus
\bigoplus_{j=1}^l Y_j(\om)$, for every $f\in V(\om)\setminus \{0\}$,
$\lim_{n \to \infty}\frac{1}{n} \log \|\mcl_\om^{(n)} f\|_{\hpt} \leq
\ka^*$, and for every $f\in Y_j(\om)\setminus \{0\}$, $\lim_{n \to
\infty} \frac{1}{n} \log \|\mcl_\om^{(n)} f\|_{\hpt} = \lam_j$.
Furthermore, the norms of the associated projections are tempered with
respect to $\sig$.
\end{thm}

\begin{rmk}
We remark that in both the scenarios that we consider, the existing proofs
of Buzzi \cite{Buzzi} and Cowieson \cite{Cowieson} establish the
existence of random absolutely continuous invariant measures. One can
check, using techniques such as those in \cite[Proposition 3.2]{Buzzi}
that their densities lie in the leading Oseledets subspace, $Y_1(\omega)$.
\end{rmk}

\section{Future work}\label{sec:futureWork}

Several interesting questions remain open for future research.
In relation to previous works concerned with exponential decay of correlations,
it is natural to look for conditions that provide further information about the structure of the Oseledets splitting, either in a general framework, or in the specific 
situation of random composition of piecewise expanding maps.
Of particular interest would be to ensure simplicity of the leading Lyapunov
exponent, and to obtain bounds on the number of exceptional Lyapunov exponents, including  finiteness.
Some
progress has already been achieved in this direction in the settings of
smooth expanding maps (see work of Baladi, Kondah and Schmitt
\cite{BaladiKondahSchmitt}) subshifts of finite type (see work of Kifer
\cite{Kifer92, Kifer08}) and piecewise smooth expanding maps of the
interval (see work of Buzzi \cite{BuzziEDC}). 

In a different direction, it would be interesting to investigate applications of the abstract semi-invertible Oseledets theorem (Theorem~\ref{thm:GranderOseledetsSplitting}) to a more general class of expanding maps, allowing for non-constant branch partitions and, more ambitiously, to piecewise hyperbolic maps. 

Finally, we hope that the constructive approach to the identification of Oseledets spaces turns out to be useful for numerical studies of non-autonomous dynamical systems. A possible alternative would be to first attempt to identify the Oseledets filtration, perhaps using some  existing numerical method. Then, one could inductively approximate Oseledets spaces by (1) fixing a sufficiently dense subset of a basis of a suitable Banach space, (2) pushing forward these elements under a numerical approximation of the transfer operator and (3) subtracting the projection along the corresponding level in the filtration to lower Oseledets spaces, previously obtained by this procedure.
 % application to LY maps

\appendix
\section{Strong measurability in separable Banach spaces}

Let $X$ be a separable Banach space and let $\mathcal B_X$ be the
standard Borel $\sigma$-algebra generated by the open subsets of
$X$. Fix a countable dense sequence $x_1,x_2,\ldots$ in $X$ for the
remainder of this appendix. It is well known that any open subset of
$X$ is a countable union of sets of the form $U_{i,j}$, where
$U_{i,j}:=\{x\in X\colon \|x-x_i\|<1/j\}$.  Hence, $\mathcal B_X$ is
countably generated.

We denote by $L(X)$ the set of bounded linear operators from $X$ to
$X$.  The strong operator topology, $\SOT(X)$, on $L(X)$ is the
topology generated by the sub-base $\{V_{x,y, \ep}=\{T\colon
\|T(x)-y\|<\epsilon\}\}$.

\begin{defn}\label{defn:strongSigmaAlgebra}
The \textbf{strong $\sigma$-algebra} on $L(X)$ is the
$\sigma$-algebra $\mathcal S$ generated by
sets of the form
$W_{x,U}=\{T\colon T(x)\in U\}$, with $x\in X$ and $U \subset X$ open.
\end{defn}

For $r\in \R$, let $L_r(X)$ denote the linear maps from $X$ to itself
with norm at most $r$. That is, $L_r=\{T\in L(X)\colon \|T\|\le r\}$.

\begin{lem}\label{lem:Smeas}
\
  \begin{enumerate}
  \item
  For every $r\in \R$, $L_r \in \mathcal S$.\label{it:BallIsSmeas}
  \item
  $\mathcal S$ is countably generated.\label{it:SCountGen}
\item An open set in the strong operator topology lies in $\mathcal
  S$. \label{it:SOTisSmeas}
  \item
  The strong $\sigma$-algebra is the Borel $\sigma$-algebra of the strong
  operator topology $\SOT(X)$.
\item An open set in $L_n(X)$ (in the relative topology) is the union
  of countably many sets of the form $B_{i,j,m,n}$ (with terminology
  introduced in the proof).\label{cor:Ln}
  \end{enumerate}
\end{lem}

\begin{proof}
  We first show that $L_r \in \mathcal S$. Let
  \begin{align*}
    \tld{L}_r&=\bigcap_j\left\{T\colon \|T(x_j)\|\le r\|x_j\|\right\}\\
      &=\bigcap_j\bigcap_k \left\{T\colon \|T(x_j)\|<(r+1/k)\|x_j\|\right\}.
  \end{align*}
  Then, $\tld{L}_r$ is a countable intersection of sets in the sub-base and
  therefore $\tld{L}_r\in\mathcal S$. We claim that $L_r=\tld{L}_r$.
Notice that if $\|T\|\le r$, then $T\in \tld{L}_r$.
  Conversely let $T\in \tld{L}_r$ and $x\in X$, let $x_j\to x$. Since $T$ is
  bounded we have $T(x_j)\to T(x)$. Since $\|T(x_j)\|\le r\|x_j\|$ and
  $\|x_j\|\to\|x\|$ we see that $\|T(x)\|\le r\|x\|$. Since this holds
  for all $x$, we see that $\|T\|\le r$. Thus, $L_r=\tld{L}_r$, as claimed.

  Set $V_{i,j,m}=\{T\colon \|T(x_i)-x_j\|< 1/m\}$. This clearly
  belongs to $\mathcal S$. We claim that an open set $U\subset L(X)$
  in the strong operator topology is the union of sets of the form
  $B_{i,j,m,n}=V_{i,j,m}\cap L_n$. Let $U$ be open and let $T\in U$.
  Then $U$ contains a basic open neighbourhood of $T$, that is, a set
  of the form $\{S\colon \|S(y_i)-T(y_i)\|<\epsilon_i\text{ for
  }i=1,\ldots,s\}$ (where $y_1,\ldots,y_s$ are elements of $X$). Let
  $n>\|T\|$ and let $m>\max(3/\epsilon_i)$. Choose $x_{k_i}$ such that
  $\|x_{k_i}-y_i\|<\min(1/(2mn),\epsilon_i/(3n))$ and $x_{\ell_i}$
  such that $\|x_{\ell_i}-T(y_i)\|<\min(\epsilon_i/3,1/(2m))$. Let
  $C=\bigcap_{i=1}^s B_{k_i,\ell_i,m,n}$.  By an application of the
  triangle inequality if $S\in C$, then we have
  \begin{align*}
    \|S(y_i)-T(y_i)\|&\le
    \|S(y_i)-S(x_{k_i})\|+\|S(x_{k_i})-x_{\ell_i}\|+\|x_{\ell_i}-T(y_i)\|\\
    &\le n\|y_i-x_{k_i}\|+1/m+\epsilon_i/3<\epsilon_i,
  \end{align*}
  so that we see $C\subseteq U$. We also have
  \begin{align*}
    \|T(x_{k_i})-x_{\ell_i}\|&\le
    \|T(x_{k_i})-T(y_i)\|+\|T(y_i)-x_{\ell_i}\|\\
    &\le \|T\|\|x_{k_i}-y_i\|+1/(2m)\le 1/m,
  \end{align*}
  so that $T\in C$. It follows that any open set $U$ may be expressed
  as a countable union of finite intersections of $\mathcal
  S$-mea\-sur\-able sets of the form $B_{i,j,m,n}$, so that open sets
  belong to $\mathcal S$, proving~\eqref{it:SOTisSmeas}.
  \eqref{cor:Ln} is proved similarly.

  Since $\mathcal S$ is generated by sets that are open in the strong
  operator topology, it follows that $\mathcal S$ is also generated by
  the $(B_{i,j,m,n})$, proving~\eqref{it:SCountGen}.

  We have shown that $\mathcal S$ contains all open sets in the strong
  operator topology. By definition it is generated by a collection of
  sets that are open in the strong operator topology. It follows that
  $\mathcal S$ is the Borel $\sigma$-algebra of $\SOT(X)$.

\end{proof}

\begin{defn}\label{defn:stronglyMble}
  A map $T:\Om \to L(X)$ is called \textbf{strongly measurable} if for
  every $x\in X$, the map $T(\cdot)(x):\Om \to X$ given by $\om
  \mapsto T(\om)(x)$ is $(\mc{F}, \mc{B}_X)$ measurable.
\end{defn}

\begin{lem}\label{lem:equivStronglyMble}
  $T:\Om \to L(X)$ is strongly measurable if and only if it is
  $(\mc{F}, \mc{S})$-mea\-sur\-able.
\end{lem}
\begin{proof}
  Recall that since $X$ is separable, both $\mc{B}_X$ and $\mc{S}$ are
  countably generated Borel $\sigma$-algebras. $\mc{B}_X$ is generated
  by $U_{i,j}:=\{x\in X\colon \|x-x_i\|<1/j\}$ (where $\{x_i\}_{i \in
    \N}$ is a dense set in $X$) and in view of Lemma~\ref{lem:Smeas},
  $\mc{S}$ is generated by the sets $B_{i,j,m,n}=V_{i,j,m}\cap L_n$
  (where $V_{i,j,m}=\{T\colon \|T(x_i)-x_j\|< 1/m\}$).  Furthermore,
  it is straightforward to check that
\begin{equation}\label{eq:strMble}
  T^{-1}(V_{x,y,\ep})=\{\om : |T(\om)(x)-y|<\ep\}= T(\cdot)(x)^{-1}(B_\ep(y)),
\end{equation}
where $V_{x,y, \ep}=\{T\colon \|T(x)-y\|<\epsilon\}$.

Assume that $T:\Om \to L(X)$ is strongly measurable. To show it is
$(\mc{F}, \mc{S})$ measurable, it suffices to show that
$T^{-1}(B_{i,j,m,n}) \in \mc{F}$. This follows from \eqref{eq:strMble}
and the fact that $L_n =\bigcap_j\bigcap_k \left\{T\colon
  \|T(x_j)\|<(n+1/k)\|x_j\|\right\}$, which was established in the
proof of Lemma~\ref{lem:Smeas}\eqref{it:BallIsSmeas}.

For the converse, suppose that $T:\Om \to L(X)$ is $(\mc{F}, \mc{S})$
measurable. To show it is strongly measurable, we have to show that
for every $x \in X$ and $i, j \in \N$, $T(\cdot)(x)^{-1}(U_{i,j})\in
\mc{F}$.  Equation~\eqref{eq:strMble} gives that
$T(\cdot)(x)^{-1}(U_{i,j})=T^{-1}(V_{x,x_i,1/j})$.  Since
$V_{x,x_i,1/j}\in \mc{S}$ by
Lemma~\ref{lem:Smeas}\eqref{it:SOTisSmeas}, the result follows.
\end{proof}

\begin{lem}\label{lem:CompStronglyMeas}
The composition of strongly measurable maps is strongly measurable.
\end{lem}
\begin{proof}
  In view of Lemma~\ref{lem:equivStronglyMble}, it suffices to show
  that the composition map $\Psi: L(X) \times L(X) \to L(X)$ given by
  $\Psi(T,S)= T\circ S$ is $(\mc{S}\otimes \mc{S}, \mc{S})$
  measurable.  We claim that for every $n\in \N$, the restriction of
  $\Psi$ to $L_n(X) \times L(X)$ is continuous with respect to
  $(\tau_n(X), \SOT(X))$, where $\tau_n(X)$ is the product topology on
  $L_n(X)\times L(X)$, and where $L_n(X)$ is endowed with the subspace
  topology of $\SOT(X)$. Since $L(X)$ is $\bigcup_{n\in \N} L_n(X)$, the
  result then follows from Lemma~\ref{lem:Smeas}.

  By Lemma~\ref{lem:Smeas}\eqref{it:SCountGen}, the claim will follow
  from showing that for every $x, y, \ep$, the set $\Psi^{-1}(V_{x,y,
    \ep})\cap (L_n(X) \times L(X))$ lies in $\tau_n(X)$.  Let $(T_0,
  S_0)\in \Psi^{-1}(V_{x,y,\ep}) \cap (L_n(X) \times L(X))$. Then,
  $\|T_0\circ S_0(x)-y\|<\ep$. Let $\del<\ep-\|T_0\circ
  S_0(x)-y\|$. Then, for $(T,S)\in \big(V_{S_0(x),T_0\circ S_0(x),
    \frac{\del}{2}} \cap L_n(X) \big) \times V_{x, S_0(x),
    \frac{\del}{2n}}$, we have
\begin{align*}
  \|T\circ S(x)-y\| &\leq \|T\circ S(x)-T \circ S_0(x)\| + \|T\circ
  S_0(x)-T_0 \circ S_0(x)\| + \|T_0 \circ S_0(x)-y\| \\
  & < n \frac{\del}{2n}+ \frac{\del}{2} + \|T_0 \circ S_0(x)-y\|< \ep.
\end{align*}
Thus, $\Psi(T,S)\in V_{x,y,\ep}$ and $\Psi^{-1}(V_{x,y, \ep})$ is open
in $\tau_n(X)$, as claimed.
\end{proof}

Let $\Phi\colon L(X)\times X\to X$ be given by $(T,x)\mapsto
T(x)$.

\begin{lem}
  \ \begin{enumerate}
  \item The restriction of $\Phi$ to $L_n(X)\times X$ is continuous,
    where $L_n(X)$ is endowed with the subspace topology of
    $SOT(X)$.\label{it:cty}
  \item $\Phi$ is $\mathcal S\times \mathcal B_X$-mea\-sur\-able.
  \label{it:mble}
\item If $\tau:\Om \to L(X)$, given by $\omega\mapsto T_\omega$, is
  strongly measurable and $f:\Om \to X$, given by $\omega\mapsto
  x_\omega$, is measurable, then $\omega\mapsto T_\omega(x_\omega)$ is
  measurable.\label{it:comp}
  \end{enumerate}
\end{lem}

\begin{proof}
  Let $U$ be an open subset of $X$ and let $A=\Phi^{-1}U\cap
  (L_n(X)\times X)$. Let $(T,x)\in A$, so that $T(x)\in U$ and $T\in
  L_n(X)$. Since $U$ is open there exists an $\epsilon>0$ such that
  $B_\epsilon(T(x))\subset U$.

  Now let $N=\{S\in L_n(X)\colon \|S(x)-T(x)\|\le \epsilon/2\}$. If
  $(S,y)\in N\times B_{\epsilon/(2n)}(x)$ then we see
  \begin{align*}
    \|\Phi(S,y)-\Phi(T,x)\|&=\|S(y)-T(x)\|\le
    \|S(y)-S(x)\|+\|S(x)-T(x)\|\\
    &\le n\|y-x\|+\epsilon/2<\epsilon.
  \end{align*}
  It follows that $N\times B_{\epsilon/(2n)}(x)$ is a subset of $A$ so
  that $A$ is open in the relative topology on $L_n(X)\times X$, hence
  proving \eqref{it:cty}.

  If $U$ is open, then $\Phi^{-1}(U)=\bigcup_n \Phi^{-1}(U)\cap
  (L_n(X)\times X)$.  By the above, the set $\Phi^{-1}(U)\cap
  (L_n(X)\times X)$ is open in the relative topology on $L_n(X)\times
  X$. Since $L_n(X)\times X$ has a countable neighbourhood basis with
  respect to the relative topology (see Corollary \ref{cor:Ln}),
  $\Phi^{-1}(U)$ may be expressed as the countable union of products
  of $\mathcal S$-mea\-sur\-able sets with $\mathcal B_X$ measurable
  sets.  Therefore it is $\mathcal S\times \mathcal
  B_X$-mea\-sur\-able, proving \eqref{it:mble}.

  Let $\tau\colon \Omega\mapsto L(X)$ be strongly measurable and
  $f\colon \omega\mapsto x_\omega$ be measurable. Let
  $\theta(\omega)=(T_\omega,x_\omega)$.  Then the map $\omega\mapsto
  T_\omega(x_\omega)$ may be factorized as $\Phi\circ\theta$. It is
  therefore sufficient to show that $\theta^{-1}\Phi^{-1}U$ is
  measurable for any open set $U$ in $X$.  We showed above that
  $\Phi^{-1}U$ is $\mathcal S\times\mathcal B_X$-mea\-sur\-able so it
  suffices to show that $\theta$ is measurable. By definition,
  $\mathcal S\times\mathcal B_X$ is generated by sets of the form
  $A\times B$ with $A\in\mathcal S$ and $B\in\mathcal B_X$. The
  preimage of $A\times B$ under $\theta$ is $\tau^{-1}(A)\cap f^{-1}B$
  which is, by assumption, the intersection of two measurable
  sets. Hence $\mathcal S\times\mathcal B_X$ is generated by a
  collection of sets whose preimages under $\theta$ are measurable and
  hence $\theta$ is measurable.

\end{proof}
 % strong measurability
\section{The Grassmannian of a Banach space}

This appendix collects some results about Grassmannians that we need
in Section~\ref{sec:OseledetsSplitting}.

Let $X$ be a Banach space. A closed subspace $Y$ of $X$ is called
\emph{complemented} if there exists a closed subspace $Z$ such that
$X$ is the topological direct sum of $Y$ and $Z$, written $X=Y\oplus
Z$. That is, for every $x\in X$, there exist $y\in Y$ and $z\in Z$
such that $x=y+z$, and this decomposition is unique.  The Grassmannian
$\mathcal G(X)$ is the set of closed complemented subspaces of $X$. We
denote by $\mathcal G^k(X)$ the collection of closed $k$-codimensional
subspaces of $X$ (these are automatically complemented). We denote by
$\mathcal G_k(X)$ the collection of $k$-dimensional subspaces of $X$
(these are automatically closed and complemented).  We equip $\mathcal
G(X)$ with the metric $d(Y,Y')=d_H(Y\cap B,Y'\cap B)$ where $d_H$
denotes the Hausdorff distance and $B$ denotes the closed unit ball in
$X$. We let $B^*$ denote the closed unit ball in $X^*$. We denote by
$\mc{B}_{\mc{G}}$ the Borel $\sigma$-algebra coming from $d$.

There is a natural map $\perp$ from $\mathcal G(X)$ to $\mathcal
G(X^*)$, namely $Y^\perp=\{\theta\in X^*\colon \theta(y)=0\text{ for
  all $y\in Y$}\}$. We use the same notation for the map from
$\mathcal G(X^*)$ to $\mathcal G(X)$ given by $W^\perp=\{y\in Y\colon
\theta(y)=0\text{ for all $\theta\in W$}\}$. Notice that if $X$ is a
reflexive Banach space, then the two notions of $\perp$ on $X^*$
agree. It is well known that for a closed subspace $Y$ of $X$,
$Y^{\perp\perp}=Y$ (that $Y\subseteq Y^{\perp\perp}$ follows from the
definitions; that $Y^{\perp\perp}\subseteq Y$ follows from the
Hahn-Banach theorem).

The following result may be found in Kato \cite[IV \S2]{Kato}.
\begin{lem}\label{lem:perp}
  The maps $\perp$ from $\mathcal G(X)$ to $\mathcal G(X^*)$ and
  from $\mathcal G(X)$ to $\mathcal G(X^*)$
  are homeomorphisms.
\end{lem}

\begin{defn}\label{def:nicebasis}
If $Y\in \mathcal G_k(X)$, we will say a basis $\{y_1,\ldots,y_k\}$ for
$Y$ is a \emph{nice basis} if $\|y_i\|=1$ and
$d(y_i,\text{span}(y_1,\ldots,y_{i-1})=1$ for each $i$. A subset of size
$k$ will be called $\epsilon$-\emph{nice} if
$1-\epsilon<\|y_i\|<1+\epsilon$ and
$d((y_i,\text{span}(y_1,\ldots,y_{i-1}))>1-\epsilon$ for each $i>1$.
Clearly if $\epsilon<1$ an $\epsilon$-nice set is linearly independent.
If a set is $\epsilon$-nice and a basis, we call it an
$\epsilon$-\emph{nice basis}.
\end{defn}

\begin{lem}\label{lem:nicebasis} Each element of $\mathcal G_k(X)$ has
  a nice basis.
\end{lem}

\begin{proof}
  Let $Y\in \mathcal G_k(X)$. We make the inductive claim that for
  each $m\le k$ we can find a sequence of elements $y_1,\ldots,y_m$ of
  norm 1 satisfying the claim for $1\le i\le m$.  The base case is
  easy: let $y_1$ be any vector in $Y$ of length 1.  Suppose we have
  vectors $y_1,\ldots,y_m$ satisfying the claim where $m<k$. Let $W$
  be the subspace of $Y$ spanned by $y_1,\ldots,y_m$. Let $x\in
  Y\setminus W$ (such an $x$ exists because $W$ is an $m$-dimensional
  subspace of $Y$ and hence a proper subspace of $Y$). By compactness
  there is a $w\in W$ minimizing $\|x-w\|$. Then let $y_{m+1}$ be a
  normalized version of $x-w$. This completes the inductive step and
  hence the proof.
\end{proof}

\begin{lem}\label{lem:coords}
Let $\{y_1,\ldots,y_k\}$ be $\epsilon$-nice (with $\epsilon<2^{-k-2}$).
If $\|\sum a_iy_i\|\le 1$ then $|a_i|\le 2^{k+1-i}$ for each $i$.
\end{lem}

\begin{proof}
Suppose for a contradiction that $y=\sum a_iy_i$ satisfies $\|y\|=1$\ and
$|a_i|>2^{k+1-i}$ for some $i$. Let $i$ be the largest such index. Set
$z=\sum_{j\le i}a_jy_j$. Then $\|z-y\|\le \sum_{j>i}|a_j|(1+\epsilon)\le
(2^{k+1-i}-2)(1+\epsilon)$. On the other hand by the defining property of
$y_i$ we have $\|z\|\ge |a_i|(1-\epsilon)^2>2^{k+1-i}(1-2\epsilon)$. The
triangle inequality then shows that $\|y\|>2-2^{k+3-i}\epsilon>1$, which
contradicts the assumption.
\end{proof}

\begin{lem}\label{lem:close}
Let $y_1,\ldots,y_k$ be an $\epsilon$-nice basis for a $k$-dimensional
space $Y$ (with $\epsilon<2^{-k-2}$), then if $W$ is a space such that
$d(y_i,W)<\delta/2^{k+2}$, then $\sup_{y\in Y\cap B}d(y,W\cap B)<\delta$.
\end{lem}

\begin{proof}
  Let $d(y_i,w_i)<\delta/2^{k+2}$. Given $y\in Y\cap B$, $y$ may be
  expressed as $\sum a_iy_i$ with $|a_i|\le 2^{k+1-i}$, by
  Lemma~\ref{lem:coords}. Let $w=\sum a_iw_i$. Then $\|y-w\|\le \sum
  2^{k+1-i}\delta/2^{k+2}<\delta/2$. It follows that $\|w\|<
  1+\delta/2$ so letting $w'=w$ if $\|w\|\le 1$ and $w/\|w\|$
  otherwise, we have $\|w'-w\|<\delta/2$. Since $w'\in W\cap B$ we
  deduce $d(y,W\cap B)< \delta$ so that $\sup_{y\in Y\cap B}d(y,W\cap
  B)<\delta$ as required.
\end{proof}

For $\epsilon<2^{-k-2}$, let $NB_k^\ep(X)\subset X^k$ be the set of
$k-$dimensional $\ep-$nice subsets of $X$.
\begin{cor}\label{cor:niceBasisToGrassmannian}
If $\ep<2^{-k-2}$, then the function from $NB_k^\ep(X)$ to $\mc{G}_k(X)$
given by $(y_1,\dots, y_k) \mapsto \text{span}(y_1, \dots, y_{k})$ is
continuous.
\end{cor}
\begin{proof}
  By Lemma~\ref{lem:close}, if $(y_1,\dots, y_k), (y_1',\dots,
  y_k')\in NB_k^\ep(X)$ are such that $\|y_i-y_i'\|< 2^{-(k+2)}\del$,
  then $d(Y,Y')<\del$, where $Y=\text{span}(y_1, \dots, y_k)$ and
  $Y'=\text{span}(y'_1, \dots, y'_k)$.
\end{proof}

\begin{lem}\label{lem:close2}(Symmetry of closeness in $\mathcal
  G_k(X)$ and $\mathcal G^k(X)$).
  Let $Y$ and $W$ be elements of $\mathcal G_k(X)$. Suppose that
  $\max_{y\in Y\cap B}d(y,W\cap B)=r<3^{-k}/4$. Then one obtains
  $\max_{w\in W\cap
    B}d(w,Y\cap B)<4\cdot 3^kr$ and hence $d(W,Y)<4\cdot 3^kr$.

  Let $Y$ and $W$ be elements of $\mathcal G^k(X)$. Suppose that
  $\max_{y\in Y\cap B}d(y,W\cap B)=r<3^{-k}/8$. Then, $d(W,Y)<8\cdot
  3^kr$.
\end{lem}

\begin{proof}
  Using Lemma \ref{lem:nicebasis}, let $y_1,\ldots,y_k$ be a nice
  basis for $Y$. By assumption there exist elements $w_1,\ldots,w_k$
  of $W\cap B$ such that $\|w_i-y_i\|<r$.

  We first give a lower bound for $\|\sum_{i=1}^ka_iw_i\|$. Let
  $M=\max_i 3^i|a_i|$ and assume that $M=3^j|a_j|$.  Then we have
  \begin{align*}
    \left\|\sum_{i=1}^{k}a_iw_i\right\|&\ge
    \left\|\sum_{i=1}^ja_iw_i\right\|-\left\|\sum_{i=j+1}^ka_iw_i\right\|\\
    &\ge \left\|\sum_{i=1}^j a_iy_i\right\|-\left\|\sum_{i=1}^j
      a_i(w_i-y_i)\right\|
    -\sum_{i=j+1}^k 3^{-i}M\\
    &\ge |a_j|-r\sum_{i=1}^j |a_i|-3^{-j}M/2\\
    &\ge M(3^{-j}/2-r/2)\ge (3^{-k}/4)M,
  \end{align*}
  where for the first term of the third inequality we used the
  definition of nice basis. In particular if $\sum a_iw_i\in B$, we
  have $|M|\le 4\cdot 3^k$ so that $|a_i|\le 4\cdot 3^{k-i}$. Now
  given $w\in W\cap B$, write $w=\sum a_iw_i$ and let $y=\sum
  a_iy_i$. Then $\|w-y\|\le \sum |a_i|\|w_i-y_i\|\le 2\cdot
  3^kr$. Rescaling $y$ to move it inside $B$ if necessary we obtain a
  point $y'\in Y\cap B$ with $\|w-y'\|\le 4\cdot 3^kr$. It follows
  that $\max_{w\in W\cap B}d(w,Y\cap B)\le 4\cdot 3^kr$ as required.

  A similar argument using the proof of Lemma~\ref{lem:perp} gives the
  result in $\mathcal G^k$.
\end{proof}

\begin{lem}\label{lem:niceBasesOpen}
If $\ep<2^{-k-2}$ , the set $NB_k^\ep(X)\subset X^k$ is open.
\end{lem}

\begin{proof}
Let $(y_1,\dots, y_k)\in NB_k^\ep(X)$. Let $\ep'<\ep$ be such that
for every $1 \leq j \leq k$, $1-\ep' \leq \|y_j\| \leq 1+\ep'$ and
$d(y_j, \text{span}(y_1, \dots, y_{j-1}) \geq 1-\ep'$.

Let $r>0$ and $(w_1, \dots, w_k) \in \Pi_{i=1}^k B_r(y_i)$.  Then, for
every $1 \leq j \leq k$ we have $1-\ep'-r \leq \|w_j\| \leq 1+\ep'+r$
and, if $r>0$ is sufficiently small, Lemmas~\ref{lem:close} and
\ref{lem:close2} imply that $d(\text{span}(y_1, \dots, y_j),
\text{span}(w_1, \dots, w_j)) \leq 2^{j+4} 3^j r$.

It follows from the triangle inequality that
\begin{align*}
  d(w_j, \text{span}(w_1, \dots, w_{j-1}))\geq &
  d(y_j,\text{span}(y_1, \dots, y_{j-1})) - \|y_j-w_j\| \\
  & - 2\|w_j\|d(\text{span}(y_1, \dots, y_j), \text{span}(w_1, \dots,
  w_j)).
\end{align*}
Hence, $d(w_j, \text{span}(w_1, \dots, w_{j-1}))\geq 1-\ep' - r -
(1+\ep'+r)2^{j+5} 3^j r$.  Thus, choosing $r$ sufficiently small, the
above yields that $(w_1, \dots, w_k)$ satisfies all the conditions of
an $\epsilon$-nice basis.
\end{proof}

\begin{lem}\label{lem:discGrass} (Disconnectedness of $\mathcal
  G(X)$). If $Y\in \mathcal G_j(X)$ and $\text{dim}(Y')>j$ then
  $d(Y,Y')\ge 2^{-j}/8$.  If $Y\in \mathcal G^j(X)$ and
  $\text{codim}(Y')>j$ then $d(Y,Y')\ge 2^{-j}/16$.
\end{lem}

\begin{proof}
  We have
  \begin{equation*}
    d(Y,Y')=\max(\max_{y\in Y\cap B}d(y,Y'\cap B),\max_{y'\in
      Y'\cap B}d(y',Y\cap B)).
  \end{equation*}
  Suppose that the first term is less than $2^{-j}/8$. Let
  $y_1,\ldots,y_j$ be a nice basis for $Y$ (as provided by Lemma
  \ref{lem:nicebasis}). Let $\|w_i-y_i\|<2^{-j}/8$ and let $W$ be the
  space spanned by the $w_i$. Lemma \ref{lem:close} guarantees that
  $\sup_{y\in Y\cap B}d(y,W\cap B)<1/2$. Since $W$ is at most
  $j$-dimensional whereas $\text{dim}(Y')>j$, let $z\in Y'\cap B$
  satisfy $d(z,W)=1$. Now for $y\in Y\cap B$, $d(z,W)\le
  d(z,y)+d(y,W)$, so using $d(y,W)\le 1/2$ we see that $d(z,y)\ge
  1/2$. We have therefore shown that $\max_{y'\in Y'\cap B}d(y',Y)\ge
  1/2$ under the assumption that the first term of $d(Y,Y')$ is
  small. In either case we see that $d(Y,Y')$ is bounded below by
  $2^{-j}/8$.  The second statement follows from Lemma~\ref{lem:perp}.
\end{proof}

\begin{cor}\label{cor:oneWayClose}
  Suppose that $Y\in \mathcal G_k(X)$, $Y'\in \mathcal G_{k'}(X)$ and
  $\sup_{y\in Y\cap B}d(y,Y'\cap B)<\ep$.  Then, if $\ep$ is
  sufficiently small, we have that $k'\geq k$.  Similarly, suppose
  that $Y\in \mathcal G^k(X)$, $Y'\in \mathcal G^{k'}(X)$ and
  $\sup_{y\in Y\cap B}d(y,Y'\cap B)<\ep$. Then, if $\ep$ is
  sufficiently small, we have that $k' \leq k$.
\end{cor}

\begin{proof}
We present the proof of the second statement, which is the one used.
The proof of the first one is entirely analogous.

Let $\ep<3^{-k}/8$, and assume the hypotheses hold.  We want to show that
$k'\leq k$. Assume on the contrary that $k'>k$, and pick $\tld{Y}\subset
Y$ such that $\tld{Y}\in \G^{k'}(X)$. Then, $\sup_{y\in
  \tld{Y}\cap B}d(y,Y'\cap B)\leq \sup_{y\in Y\cap B}d(y,Y'\cap
B)<\ep$. In view of Lemma~\ref{lem:close2}, $d(\tld{Y}, Y') <3^{k'} 8
\ep$, so $\sup_{y'\in Y'\cap B}d(y',Y\cap B)\leq \sup_{y'\in Y'\cap
  B}d(y',\tld{Y}\cap B)< 3^{k'} 8 \ep$.  So $d(Y,Y')<3^{k'} 8
\ep$. If $\ep$ is sufficiently small, this contradicts
Lemma~\ref{lem:discGrass}.  Thus, $k'\leq k$.
\end{proof}

\begin{lem} \label{lem:G_kseparable}
  If $X$ is separable, then $\mathcal {G}_k(X)$ is separable.
\end{lem}

\begin{proof}
  Let $x_1,x_2,\ldots$ be a dense sequence in $B$. Then the collection
  of linearly independent $k$-element subsets of $\{x_1,x_2,\ldots\}$
  is also countable. Let $Y\in \mathcal G_k(X)$. Let $y_1,\ldots,y_k$
  be a nice basis for $Y$ (as in Lemma \ref{lem:nicebasis}). Then
  given $\epsilon>0$, let $(x_{i_j})_{j=1}^k\in B$ be chosen so that
  $\|x_{i_j}-y_j\|\le \epsilon/2^{k+2}$. Let
  $W=\text{span}(x_{i_1},\ldots,x_{i_k})$. Lemma \ref{lem:close}
  implies that $\max_{y\in Y\cap B}d(y,W)<\epsilon$. Then Lemma
  \ref{lem:close2} implies that $d(Y,W)\le 4\cdot 3^{k}\epsilon$,
  so the separability is established.
\end{proof}

\begin{lem}\label{lem:G^kseparable}
  If $X^*$ is separable, then $\mathcal G^k(X)$ is separable.
\end{lem}

\begin{proof}
  Lemma~\ref{lem:perp} implies that $\mathcal G^k(X)$ is homeomorphic
  to $\mathcal G_k(X^*)$. The result follows from
  Lemma~\ref{lem:G_kseparable}.
\end{proof}

\begin{lem}\label{lem:ImagePWCont}
Let $\Phi: L(X)\times \mc{G}(X) \to \mc{G}(X)$ be given by
$\Phi(T,W)=T(W)$. For $0\leq l \leq k$, let
\[
\tld{G}(n,k, l)=\{(T, W) \in L_n(X)\times \mc{G}_k(X): \dim (Ker\, T \cap
W) \geq l\}.
\]
This is a closed subset of $L_n(X)\times \mc{G}_k(X)$, where $L_n(X)$
is endowed with the restriction of the strong operator topology on
$L(X)$. Also, let
\begin{align*}
G(n,k, l)&=\tld{G}(n,k, l)\setminus \tld{G}(n,k, l+1)\\
&=\{(T, W) \in L_n(X)\times \mc{G}_k(X): \dim (Ker\, T \cap W)= l\}.
\end{align*}
Then, $\Phi|_{G(n,k, l)}: G(n,k, l) \to \mc{G}_{k-l}(X)$ is continuous.
\end{lem}

\begin{proof}
  To see that $\tilde G(n,k,l)$ is a closed set, suppose that $T\in
  L_n(X)$, $W\in \mc{G}_k(X)$ and $\dim(\text{Ker} T\cap W)=s<l$. Let
  $\{Tw_1,\ldots,Tw_{k-s}\}$ be a nice basis for $T(W)$. Let
  $M=\max\|w_i\|$. Let $U_\del=\{S\in L_n(X)\colon
  \|S(w_i)-T(w_i)\|<\delta\}$ (a relatively open subset of $L_n(X)$),
  and let $V\in \mc{G}_k(X)$ be such that $d(W,V)<\delta$. Then in
  particular $V$ contains elements $v_1,\ldots,v_{k-s}$ such that
  $\|v_i-w_i\|<M\delta$. Now we have $\|S(v_i)-T(w_i)\|\le
  \|S\|\|v_i-w_i\|+\|S(w_i)-T(w_i)\|\le (nM+1)\delta$.  By
  Lemma~\ref{lem:niceBasesOpen}, if $\ep>0$ and $\del$ is
  small enough, then $\{Sv_1,\ldots,Sv_{k-s}\}$ is an $\ep$-nice basis of
  $\text{span}(Sv_1,\ldots,Sv_{k-s})\subset S(V)$.  In particular,
  $S(V)$ has dimension at least $k-s$.  It follows that
  $\dim(\text{Ker} S\cap V)\le s$, so $\tilde G(n,k,l)^c \subset
  L_n(X)\times \mc{G}_k(X)$ is open.

  Let $(T,W) \in G(n,k, l)$, and let $r>0$. Let
  $\{Tw_1,\ldots,Tw_{k-l}\}$ be a nice basis for $T(W)$, and let $M$
  and $U_\del$ as in the previous paragraph. We claim that if $\del>0$
  is sufficiently small and $(S, V)\in \Big(U_\del \times
  B_{\mc{G}}(W, \del) \Big) \cap G(n,k,l)$, then $d(T(W),
  S(V))<r$. Indeed, let $\ep>0$. The previous argument shows that if
  $\del$ is sufficiently small and $\|v_i-w_i\|<M\delta$ for each
  $1\leq i \leq k-l$, then $\{S v_1,\dots,S v_{k-l}\}$ is an
  $\ep$-nice basis for $S(V)$ such that $\|S(v_i)-T(w_i)\|\le
  (nM+1)\delta$. Hence, by the proof of
  Corollary~\ref{cor:niceBasisToGrassmannian}, $d(T(W),
  S(V))<2^{k+2}(nM+1)\del$. Taking $\del<\frac{r}{2^{k+2}(nM+1)}$
  yields the claim.
\end{proof}

\begin{cor}\label{cor:ImageIsMble}
The map $\Phi_k:L(X) \times \mathcal G_k(X) \to \mathcal G_{\leq k}(X)$
given by  $\Phi_k(T, W)= T(W)$ is $(\mathcal S \otimes \mathcal
B_{\mathcal G}, \mathcal B_{\mathcal G})$-mea\-sur\-able.
\end{cor}
\begin{proof}
  Let us note that $ L(X) \times \mc{G}_k(X) =\bigcup_{n\in \N}
  \bigcup_{l=0}^k G(n,k,l)$. Also, $G(n,k,l)$ is $\mathcal S \otimes
  \mathcal B_{\mathcal G}$ measurable, as $L_n(X)$ is $\mathcal S $
  measurable, and $G(n,k,l)$ is the difference of two closed sets in
  $L_n(X)\times \mc{G}_k(X)$, by Lemma~\ref{lem:ImagePWCont}. Since
  $\Phi|_{G(n,k,l)}$ is continuous again by
  Lemma~\ref{lem:ImagePWCont}, then $\Phi_k$ is $(\mathcal S \otimes
  \mathcal B_{\mathcal G}, \mathcal B_{\mathcal G})$-mea\-sur\-able.
\end{proof}

\begin{lem}\label{lem:ImageProj}
  Let $X$ be separable, $\Pi: X \to Y$ a surjective bounded linear
  map, and $k\geq 0$. Then, the restriction of the induced map $\Pi_k:
  \mc{G}^k(X) \to \mc{G}(Y)$ to $\Pi^{-1}(\mc{G}^{j}(Y))$ is
  continuous for every $j\geq 0$. Furthermore, $\Pi_k$ is measurable.
\end{lem}

\begin{proof}
  First, we note that since $\Pi$ is surjective, $Y$ is separable and
  for every $W\in \mc{G}^k(X)$, $\Pi_k(W)\in \mc{G}^{\leq k}(Y)$.

  For every $0\leq j$, the set $\G^k(X) \cap \Pi_k^{-1}(\mc{G}^{\leq
    j}(Y))$ is relatively open in
  $\G^k(X)$. \label{it:InvImgOpenInG^k} To see this, let $W\in \G^k(X)
  \cap \Pi_k^{-1}(\mc{G}^{j}(Y))$. By the above, $j\leq k$. Let $W'\in
  \G^k(X)$ be such that $d(W,W')\leq \del$. By the open mapping
  theorem, there exists $r>0$ such that $rB_{\Pi_k(W)} \subset
  \Pi_k(W\cap B_X)$. Let $z\in \Pi_k(W) \cap B_Y$. Then, there exist
  $w\in W\cap \frac{1}{r}B_X$ and $w'\in W'\cap \frac{1}{r}B_X$ such
  that $\Pi w=z$ and $\|w-w'\|\leq \frac{\del}{r}$. Hence, $\|\Pi w-
  \Pi w'\|\leq \frac{\|\Pi\| \del}{r}$. Thus,
\[
\sup_{z\in \Pi_k(W)\cap B_Y}d(z,\Pi_k(W') \cap B_Y)\leq 2 \sup_{z\in
  \Pi_k(W)\cap B_Y}d(z,\Pi_k(W'))<\frac{2\|\Pi\| \del}{r}.
\]
If $\del$ is sufficiently small, Corollary~\ref{cor:oneWayClose}
implies that $W' \in \Pi_k^{-1}(\mc{G}^{\leq j}(X))$.

It follows from the proof in the previous paragraph and
Lemma~\ref{lem:close2} that the restriction of $\Pi_k$ to
$\Pi_k^{-1}(\mc{G}^{j}(Y))$ is continuous.

The fact that $\Pi_k: \mc{G}^k(X) \to \mc{G}(Y)$ is measurable follows
from the previous two paragraphs.
\end{proof}

\begin{lem}\label{lem:normRestriction}
  The function $\nu: L(X) \times \mc{G}_k(X) \to \R$ given by $\nu(R,
  Y)= \|R|_{Y}\|$ is measurable when $L(X)$ is endowed with the strong
  $\sig$-algebra $\mc{S}$.
\end{lem}
\begin{proof}
  By Lemma~\ref{lem:Smeas}, it suffices to show continuity of $\nu$ as
  restricted to $L_n(X)\times \mathcal G_k(X)$, where $L_n(X)$ is
  endowed with the restriction of the strong operator topology. Let
  $Y\in\mathcal G_k(X)$ and let $R\in L_n(X)$. Let $\epsilon>0$ and
  let $\delta<\epsilon/(n+k2^k)$. Let $\{y_1,\ldots,y_k\}$ be a nice
  basis for $Y$. Let $N=\{S\in L_n(X)\colon
  \|S(y_i)-R(y_i)\|<\delta\}$ and let $W\in \mathcal G_k(X)$ satisfy
  $d(W,Y)<\delta$. Now given $w\in B\cap W$, there exists a $y\in
  B\cap Y$ such that $\|w-y\|<\delta$ (or conversely given $y\in B\cap
  Y$, there exists a $w\in B\cap W$ such that $\|w-y\|<\delta$). We
  then have
\begin{align*}
  \|S(w)\|&\le \|R(y)\| +\|S(y)-R(y)\|+\|S(w)-S(y)\|\text{ and}\\
  \|S(w)\|&\ge \|R(y)\| -\|S(y)-R(y)\|-\|S(w)-S(y)\|.
\end{align*}

It follows that $|\|S(w)\|-\|R(y)\||\le \|S(y)-R(y)\| + \|S(w)-S(y)\|$.

The second term of the right side is bounded above by $n\delta$. For the
first term, notice that by Lemma \ref{lem:coords}, $y$ may be expressed
as a linear combination of $y_i$'s with coefficients bounded above by
$2^k$. Hence the first term is bounded above by $k2^k\delta$. It follows
that $|\|S(w)\|- \|R(y)\||\le (n+k2^k)\delta$.

By taking $y\in Y\cap B$ for which $\|R(y)\|=\|R|_Y\|$ it follows that
$\|S|_W\|>\|R|_Y\|-\epsilon$. Similarly taking $w\in W\cap B$ for which
$\|S(w)\|=\|S|_W\|$ we obtain $\|R|_Y\|>\|S|_W\|-\epsilon$ so that
$|\|R|_Y\|-\|S|_W\||<\epsilon$ as required.
\end{proof}

A pair of closed complemented subspaces $Y, Z$ of $X$ is called
\textit{complementary} if $Y\cap Z=\{0\}$ and $Y\oplus Z=X$. By the
closed graph theorem, any pair of complementary spaces $(Y, Z) $
specifies a bounded linear map $\Pi_{Y\| Z}$, which is the projection
onto $Y$ along $Z$, having kernel $Z$ and image $Y$. By symmetry, the map
$\Pi_{Z\|Y}$ is also a linear and bounded projection.

For $k\geq 0$, let
\begin{equation*}
\text{Comp}_k(X)=\{(Y,Z)\in \mc{G}_k(X) \times \mc{G}^k(X)  : Y\cap Z=\{
0 \}, Y\oplus Z =X \},
\end{equation*}
and let $\text{Comp}(X)= \bigcup_{k\geq 0}\text{Comp}_k(X)$ be the set of
complementary subspace pairs of $X$ of finite dimension/codimension.

\begin{lem}\label{lem:projAndGrass} \
Let $(Y, Z) \in \text{Comp}(X)$ and $Y' \in \mc{G}(X)$. Then,
\begin{equation*}
\| \Pi_{Z\|Y}|_{Y'} \| \leq 2\|\Pi_{Z\|Y}\| d(Y, Y').
\end{equation*}
\end{lem}

\begin{proof}
Let $y' \in Y'$ and $\ep>0$. Let $y\in Y$ be such that $\|y'-y\| \leq
d(y', Y) +\ep$. Then,
\begin{align*}
  \|\Pi_{Z\|Y} (y') \| &= \| \Pi_{Z\|Y} (y' - y) \| \leq
  \|\Pi_{Z\|Y}\| (d(y', Y) +\ep) \leq \|\Pi_{Z\|Y}\| \big( 2 \|y' \|
  d(Y', Y) +\ep).
\end{align*}
Letting $\ep\to 0$, the result follows.
\end{proof}

\begin{lem}\label{lem:ContOfProjNormTop}
  The map $\Psi: \text{Comp}(X) \to L(X)$ given by
  $\Psi(Y,Z)=\Pi_{Z\|Y}$ is continuous, where $\text{Comp}(X)$ carries
  the product topology induced by the metric on $\mc{G}(X)$ and $L(X)$
  is endowed with the norm topology.
\end{lem}
\begin{rmk}\label{rmk:ContOfProjSOTTop}
  Since the norm topology is finer than the strong operator topology
  on $L(X)$, Lemma~\ref{lem:ContOfProjNormTop} yields that $\Psi$ is
  also continuous when $L(X)$ is endowed with the strong operator
  topology.
\end{rmk}

\begin{proof}[Proof of Lemma~\ref{lem:ContOfProjNormTop}]

  Let $\ep>0$. Let $(Y, Z)\in \text{Comp}(X)$ and $x\in X$.  Since
  $\dim Y<\infty$, then $Y\cap \partial B$ is compact. Hence,
  $\zeta:=\inf_{y\in Y\cap \partial B} d(y, Z)>0$.

  Let $\del<\min\{\frac{1}{3},\frac{\zeta}{8+2\zeta}\}$ and $(Y',
  Z')\in\text{Comp}(Z)$ such that $d(Y,Y'), d(Z,Z')<\del$.  Then,
  $\inf_{y'\in Y'\cap \partial B} d(y', Z')\geq
  \frac{\zeta}{2}$. Indeed, let $y' \in Y'\cap \partial B$, and let
  $y\in Y\cap B$ be such that $\|y'-y\|<\del$. Then, $1-\del
  <\|y\|<1+\del$. Let $z\in Z$ be such that $\|y-z\|<
  d(y,Z)+\del$. Then, $\|z\|\leq 2\|y\|+\del< 3$ and
\[
  d(y', Z')\geq d(y, Z) - \|y-y'\| - d(z, Z') \geq \zeta(1-\del)- \del
  -3\del \geq \frac{\zeta}{2}.
\]
We claim that $\|\Pi_{Y'\|Z'}\|<\frac{2}{\zeta}$ and
$\|\Pi_{Z'\|Y'}\|<\frac{2}{\zeta}+1$.  Indeed, let $x\in
X\cap \partial B$, and write $x=y'+z'$, with $y'\in Y'$ and $z'\in
Z'$.  Then, $1=\|y'+z'\| \geq d(y', Z')\geq \frac{\zeta}{2}\|y'\|$, so
that $\|y'\|\leq \frac{2}{\zeta}$ and the first claim follows. The
second claim follows from the triangle inequality.

Let $M=\max\{ \|\Pi_{Y\|Z}\|, \|\Pi_{Z\|Y}\| \}$.  Assume also that
$\del< \frac{\ep}{4M}\big(\frac{2}{\zeta}+1 \big)^{-1}$.  Then, if
$(Y', Z')\in\text{Comp}(Z)$ is such that $d(Y,Y'), d(Z,Z')<\del$, we
have that
\begin{align*}
\big| \| \Pi_{Z\| Y}\| &- \| \Pi_{Z'\| Y'} \| \big| \leq \| \Pi_{Z\|
  Y} - \Pi_{Z'\| Y'} \| \\
&\leq
\| \big( \Pi_{Z\| Y} - \Pi_{Z'\| Y'} \big) |_{Z'} \| \|\Pi_{Z'\| Y'} \| +
\| \big( \Pi_{Z\| Y} - \Pi_{Z'\| Y'} \big)|_{Y'} \| \|\Pi_{Y'\| Z'}\|
\\
&\leq \| \Pi_{Y\| Z} |_{Z'} \| \|\Pi_{Z'\| Y'} \|+ \|  \Pi_{Z\|
  Y}|_{Y'} \| \|\Pi_{Y'\| Z'}\|  \\
&\leq 2M (d(Z,Z')+ d(Y,Y')) \big(\frac{2}{\zeta}+1 \big)< \ep,
\end{align*}
where the third inequality follows from the fact that $ \Pi_{Y\| Z}+
\Pi_{Z\| Y}= \mathrm{Id}$, and the fourth one follows from
Lemma~\ref{lem:projAndGrass} and the claim above.
\end{proof}

\begin{lem}\label{lem:contOfDirectSum}
Let
\begin{align*}
NI(X)=\bigcup_{k, k' \geq 0}\{(Y,Z) &\in \big(\mc{G}_k(X) \times
\mc{G}^{k'}(X) \big) \\
& \cup \big(\mc{G}^k(X) \times \mc{G}_{k'}(X) \big) \\
& \cup \big(\mc{G}_k(X) \times \mc{G}_{k'}(X) \big): Y\cap Z=\{ 0 \}\}
\end{align*}
be the set of pairs of subspaces of $X$ of finite dimension/codimension
with trivial intersection.

Then, the map $\Psi': NI(X) \to \mc{G}(X)$ be given by $\Psi'(Y,
Z)=Y\oplus Z$ is continuous.
\end{lem}
\begin{proof}
Let $(Y, Z) \in NI(X)$. Let $W \in \mc{G}(X)$ be such that $(Y \oplus Z,
W) \in \text{Comp}(X)$, so that $Y \oplus Z \oplus W= X$. Also, let
$M=\max(\|\Pi_{Y\|Z\oplus W}\|, \|\Pi_{Z\|Y\oplus W}\|)$.

Let $\del>0$ and $Y', Z' \in \mc{G}(X)$ be such that $d(Y, Y'), d(Z,
Z')<\del$. Let $y\in Y, z \in Z$ be such that $\|y+z\|\leq 1$. Then
$\|y\|, \|z\| \leq M$. Therefore, there exist $y' \in Y'$ and $z'\in Z'$
such that $\|y-y'\|, \|z-z'\|< M \del$. Hence, $\|(y+z)-(y'+z')\| \leq
\|y-y'\|+\|z-z'\| \leq 2M\del$. Therefore, $\sup_{x\in (Y\oplus Z)\cap B}
d(x, Y'\oplus Z') \leq 2M\del$, and by the triangle inequality,
$\sup_{x\in (Y\oplus Z)\cap B} d(x, (Y'\oplus Z') \cap B) \leq 4M\del$.
If $\del$ is sufficiently small, Lemma~\ref{lem:discGrass} implies that
$\text{(co)dim}Y = \text{(co)dim} Y'$ and $\text{(co)dim}Z
=\text{(co)dim} Z'$. Hence, using Lemma~\ref{lem:close2}, we get that
$d(Y \oplus Z, Y' \oplus Z')<\tld{M}\del$, where $\tld{M}$ depends on
$M$, $k$ and $k'$. Hence, $\Psi'$ is continuous.
\end{proof}
 % grassmannian
\section{Some facts from ergodic theory}
\subsection{A characterization of tempered maps}\label{sec:TempFwd=TempBwd}

This appendix provides a characterization of tempered maps, based on the
following theorem.

\begin{thm}[Tanny]
  Let $T$ be an ergodic measure-preserving transformation of a
  probability space $(X, \mathcal B,\mu)$. Let $f\colon
  X\to\mathbb{R}$ be a non-negative measurable function.  Then either
  $f(T^nx)/n\to 0$ for $\mu$-almost every $x$; or $\limsup f(T^nx)/n=\infty$
  for $\mu$-almost every $x$.
\end{thm}

The proof of the following lemma is based on a very concise proof of
Tanny's theorem, attributed to Feldman, that appears in a Lyons,
Pemantle and Peres \cite{LyonsPemantlePeres}.

\begin{lem}\label{lem:TempFwd=TempBwd}
Let $T$ be an invertible ergodic measure-preserving transformation of a
probability space $(X,\mathcal B,\mu)$. Let $f\colon X\to \mathbb R$ be a
non-negative measurable function. Then $f(T^{-n}x)/n\to 0$ for
$\mu$-almost every $x$ as $n\to\infty$ if and only if $f(T^nx)/n\to 0$
for $\mu$-almost every $x$ as $n\to\infty$.
\end{lem}

\begin{proof}
Suppose
  that $f(T^{-n}x)/n\to 0$. Let $\epsilon>0$. There exists for $\mu$-almost
  every $x$ an $L$ such that $n\ge L$ implies
  $f(T^{-n}x)/n<\epsilon$. Fixing a sufficiently large $L$, the set
  $A=\{x\colon f(T^{-n}x)/n<\epsilon\text{ for all }n\ge L\}$ has
  measure at least 1/2. Now we apply the Birkhoff ergodic theorem to
  $\mathbf 1_A$. For almost every $x$, there exists an $n_0$ such that
  for $n\ge n_0$ one has $(1/n)(\mathbf 1_A(x)+\ldots+\mathbf
  1_A(T^{n-1}x))\in[2/5,3/5]$. Fix such an $x$ and let $n_0$ be the
  corresponding quantity. Then let $N>\max(n_0,5L)$. We then have
  \begin{equation*}
  \#\{0\le i<N\colon  T^i(x)\in A\}\le 3N/5.
  \end{equation*}
  On the other hand we have
  \begin{equation*}
  \#\{0\le i<2N \colon T^i(x)\in A\}\ge 4N/5.
  \end{equation*}
  It follows that there exists $i\in [N+L,2N)$ with $T^i(x)\in A$. The
  fact that $T^i(x)\in A$ tells us that $f(T^Nx)<\epsilon
  (i-N)<\epsilon N$. It follows that $f(T^Nx)/N<\epsilon$. Since this
  holds for all large $N$ and $\epsilon$ was arbitrary we deduce that
  $f(T^nx)/n\to 0$.

The converse statement follows immediately.
\end{proof}

Combining the proof of Lemma~\ref{lem:TempFwd=TempBwd} with Tanny's
theorem, we get the following.
\begin{thm}
  Let $T$ be an invertible ergodic measure-preserving transformation
  of a probability space $(X,\mathcal B,\mu)$. Let $f\colon X\to
  \mathbb R$ be a non-negative measurable function. Then
   one of the following holds:
   \begin{itemize}
   \item $f(T^nx)/n\to 0$ for $\mu$-almost every $x$ as $n\to\pm\infty$; or
   \item $\limsup_{n\to\infty} f(T^nx)/n=\infty$ and $\limsup_{n\to\infty}
  f(T^{-n}x)/n=\infty$ for $\mu$-almost every $x$.
  \end{itemize}
\end{thm}

\begin{proof}
In view of Tanny's theorem, it is sufficient to show that if
$f(T^{-n}x)/n\to 0$ a.e. then $f(T^nx)/n\to 0$ a.e.
This follows from Lemma~\ref{lem:TempFwd=TempBwd}.
\end{proof}

\subsection{Random version of Hennion's theorem}\label{sec:randomHennion}

In this appendix, we present a result that allows us to bound the
index of compactness and maximal Lyapunov exponent of some random
dynamical systems satisfying certain Lasota-Yorke type
inequalities. We remark that many parts of this lemma essentially
appear in Buzzi \cite{Buzzi}. We have modified the conclusion and
weakened the hypotheses in one place. This result is based on the
following theorem of Hennion \cite{Hennion}.

\begin{thm}[Hennion]
  Let $(X,\|\cdot\|)$ be a Banach space and suppose that $Y$ is a
  closed subspace of $X$. Let $Y$ be equipped with a finer norm
  $\tn.\tn$ (such that $\|y\|\le \tn y\tn$ for all $y\in Y$) such that
  the inclusion of $(Y, \tn \cdot \tn)\hookrightarrow (Y, \| \cdot
  \|)$ is compact.  Suppose that $\mcl$ is a linear operator such that
  $\mcl(X)\subset X$ and $\mcl(Y)\subset Y$. Suppose further that for all
  $y\in Y$, one has the inequality
\begin{equation*}
  \tn \mcl(y)\tn\le A\| y\| + B\tn y\tn.
\end{equation*}
Then the index of compactness of $\mcl$ is bounded above by $2B$.
\end{thm}

\begin{lem}\label{lem:LY-IC-MLE}
  Let $(X,\|\cdot\|)$ be a Banach space and let $Y$ be a closed
  subspace.  Let $\tn\cdot\tn$ be a finer norm on $Y$ such that the
  inclusion of $(Y, \tn \cdot \tn)\hookrightarrow (Y, \| \cdot \|)$ is
  compact. Let $\sigma\colon(\Omega,\mu)\to(\Omega,\mu)$ be an
  invertible ergodic measure preserving dynamical system and let
  $(\mcl_\omega)_{\omega\in\Omega}$ be a family of linear maps, each
  mapping $X$ to $X$ and $Y$ to $Y$ continuously. As usual, let
  $\mcl^{(n)}_\omega=\mcl_{\sigma^{n-1}\omega}\circ\ldots\circ \mcl_\omega$.

Suppose we have the following inequalities:
\begin{align*}
  \text{(Strong L-Y)\qquad}&\tn \mcl_\omega f\tn \le A(\omega)\|
  f\|+B(\omega)\tn f\tn\text{ for all $f\in Y$};\\
  \text{(Weak L-Y)\qquad}&\tn \mcl_\omega\tn \le C(\omega),
\end{align*} where $A(\omega)$, $B(\omega)$ and $C(\omega)$ are
measurable functions, $C(\omega)$ is log-integrable and $\int\log
B(\omega)\,d\mu(\omega)<0$.

Then there exists a full measure subset $\Omega_1\subset \Omega$ with the
following properties:
\begin{enumerate}
\item $\lim_{n\to\infty}(1/n)\log\tn \mcl_\om^{(n)}
\tn_{\text{ic}}\le \int\log B(\omega)\,d\mu(\omega)$ for all
$\omega\in\Omega_1$;
\item For $\omega\in\Omega_1$, suppose that $f\in Y$ satisfies
\begin{equation}\label{eq:subexp}
\limsup_{n\to\infty}(1/n)\log \|\mcl^{(n)}_\omega f\|\le 0.
\end{equation}
Then $\limsup_{n\to\infty}(1/n)\log\tn \mcl^{(n)}_\omega f\tn \le 0$.
\end{enumerate}

\end{lem}

\begin{proof}
  For the first statement, notice that applying the strong
  Lasota-Yorke inequality we obtain inductively $\tn \mcl^{(n)}_\omega
  f\tn \le B(\sigma^{n-1})\ldots B(\omega)\tn f\tn+D\|f\|$ for a
  constant $D$ depending on $\omega$ and $n$. From Hennion's
  theorem, we deduce $\tn \mcl_\om^{(n)}\tn_{\text{ic}}\le
  2B(\sigma^{n-1}\omega)\ldots B(\omega)$.  Taking logarithms, the
  conclusion then follows from the ergodic theorem.

  We now show the second statement. There exists a $\delta>0$ such that
  for any set $S$ of measure at most $\delta$, one has $\int_S(\log
  C-\log B)\,d\mu<-\int\log B\,d\mu$. Now since $A$ is measurable,
  there exists a $K>0$ such that $\mu(\{\omega\colon A(\omega)\ge
  K\})<\delta$.

  Set $\tilde B(\omega)=B(\omega)$ if $A(\omega)\le K$ and $C(\omega)$
  otherwise. Set $\tilde A(\omega)=\min(A(\omega),K)$.  We see that we
  have a hybrid Lasota-Yorke inequality obtained by applying the
  strong Lasota-Yorke inequality for cases in which $A(\omega)\le K$
  and the weak inequality otherwise:

  \begin{equation}\label{eq:hybrid}
    \tn \mcl_\omega f\tn \le \tilde A(\omega)\|f\|+\tilde B(\omega)\tn f\tn.
  \end{equation}

  The advantage of this is that we still have $\int\log \tilde
  B\,d\mu<0$ and $\tilde A$ is now uniformly bounded by $K$.

  Applying the ergodic theorem (with the transformation being
  $\sigma^{-1}$) we obtain a measurable function $F(\omega)$ such that for
  $\bbp$-almost every $\omega$, we have
  \begin{equation}\label{eq:Aback}
    \tilde B(\sigma^{-1}\omega)\ldots\tilde B(\sigma^{-k}\omega)\le
    F(\omega)\text{ for all $k\ge 0$.}
  \end{equation}
Let $\beta>\int\log C$. Applying the ergodic theorem once more, we
obtain, for $\bbp$-almost every $\omega$, the bound
\begin{equation}\label{eq:weakit}
C(\sigma^{n+k-1}\omega)\ldots C(\sigma^n\omega)\le
H(\sigma^n\omega)e^{\beta k}\text{ for all $n,k\ge 0$}.
\end{equation}

There exists a $B>0$ such that $H(\omega)F(\omega)<B$ on a set of
positive measure. By the ergodic theorem, for all $\delta>0$, for almost
every $\omega$, there exists an $n_0$ such that
\begin{equation}\label{eq:moving}
\forall N>n_0,\ \exists n\in [N(1-\delta),N)\text{ with}
H(\sigma^n\omega)F(\sigma^n\omega)<B.
\end{equation}

Now let $\Omega_1$ be the set of full measure on which the conditions
above hold. Fix an $\omega\in\Omega_1$ and let $f\in Y$ satisfy
\eqref{eq:subexp}. Let $\epsilon>0$ be arbitrary. Then by the hypotheses,
there exists a constant $L$ such that
\begin{equation}\label{eq:coarsenormbound}
  \|\mcl_\omega^{(n)}f\|\le Le^{\epsilon n/2}\text{ for all $n\ge 0$}.
\end{equation}

Now by iterating \eqref{eq:hybrid}, we obtain the bound (valid for all
$f\in Y$)

\begin{align*}
\tn \mcl_\omega^nf\tn&\le \tilde B(\sigma^{n-1}\omega)\ldots \tilde
A(\omega)\tn f\tn+\tilde B(\sigma^{n-1}\omega)\ldots \tilde B(\sigma
\omega)\tilde A(\omega)\|f\|+\\
&\ldots+\tilde B(\sigma^{n-1}\omega)\tilde
A(\sigma^{n-2}\omega)\|\mcl_\omega^{(n-2)}f\| +\tilde
A(\sigma^{n-1}\omega)\|\mcl_\omega^{(n-1)}f\|.
\end{align*}

Using the inequalities $\tilde B(\sigma^{n-1}\omega)\ldots \tilde
B(\sigma^{n-k}\omega)\le F(\sigma^n\omega)$ (from \eqref{eq:Aback}), the
fact that $\tilde A(\omega)\le K$, and \eqref{eq:coarsenormbound}, we
obtain an upper bound of the form

\begin{equation}
\tn \mcl_\omega^n f\tn\le MF(\sigma^n\omega)e^{\epsilon n/2},
\end{equation}
for a suitable constant $M$.

Combining this with \eqref{eq:weakit} we obtain
\begin{equation}\label{eq:comb2}
\tn \mcl_\omega^{n+k}f\tn \le
MF(\sigma^n\omega)H(\sigma^n\omega)e^{\epsilon n/2}e^{\beta k}
\end{equation}

We can therefore obtain a bound for $\tn \mcl_\omega^mf\tn$ by minimizing
the above over possible decompositions $m=n+k$. Let $n_0$ be as in
\eqref{eq:moving} where $\delta$ is taken to be $\epsilon/(2\beta)$ and
suppose $m>n_0$ is given. Then there exists a $k<\epsilon/(2\beta)m$ such
that $F(\sigma^{m-k}\omega)H(\sigma^{m-k}\omega)\le B$ so that
\begin{equation*}
\tn\mcl_\omega^mf\tn\le MBe^{\epsilon m/2}e^{\beta k} < MBe^{\epsilon m}.
\end{equation*}

It follows that $\limsup_{N\to\infty} (1/N)\tn \mcl_\omega^{(N)}f\tn\le
\epsilon$. Since $\epsilon$ is arbitrary, the proof is complete.
\end{proof} % tempered + random Hennion

\bibliography{SemiInvOsel_refs}
\bibliographystyle{abbrv}

\end{document}